\documentclass{article}
\usepackage[a4paper]{geometry}
\usepackage[utf8]{inputenc}
\usepackage[all]{xy}
\usepackage[initials]{amsrefs}
\usepackage[labelsep=period]{caption}
\usepackage[fleqn]{amsmath}
\usepackage{amsfonts,authblk,bm,doi,enumitem,fancyhdr,float,graphicx,ifthen,amssymb,amsthm,mathrsfs,mathtools,mleftright,relsize,subcaption,tikz,tikz-3dplot,titlesec,wrapfig,xcolor}

\newtheorem{theorem}{Theorem}[section]

\newtheorem{lemma}[theorem]{Lemma}

\theoremstyle{definition}
\newtheorem*{definition}{Definition}

\theoremstyle{plain}
\newtheorem{innerthm}{Theorem}
\newenvironment{maintheorem}[1]
  {\renewcommand\theinnerthm{#1}\innerthm}
  {\endinnerthm}

\numberwithin{equation}{section} 

\DeclareMathOperator{\Det}{\textnormal{Det}}

\DeclareMathOperator{\sgn}{\textnormal{sgn}}

\definecolor{grey}{RGB}{210,210,210}
\definecolor{gammacol}{RGB}{48,65,93}
\definecolor{deltacol}{RGB}{207,103,102}
\definecolor{alphacol}{RGB}{3,20,36} 

\renewcommand{\geq}{\geqslant}
\renewcommand{\leq}{\leqslant}

\usetikzlibrary{calc}
\usetikzlibrary{cd}
\usetikzlibrary{matrix}
\usetikzlibrary{math}
\usetikzlibrary{decorations.markings}
\usetikzlibrary{arrows.meta}
\usetikzlibrary{arrows,shapes,positioning}
\tikzcdset{every label/.append style = {font = \normalsize}}
\tikzstyle arrowstyle=[scale=1]
\tikzstyle directed=[postaction={decorate,decoration={markings,
		mark=at position 0.5 with {\arrow[arrowstyle]{{latex}}}}}]
\tikzstyle reverse directed=[postaction={decorate,decoration={markings,
		mark=at position 0.5 with {\arrow[arrowstyle]{{latex reversed}}}}}]

\tikzset{pointer/.style 2 args={draw,fill,single arrow,
    single arrow tip angle=45,
    scale=#1,
    single arrow head indent=0pt,
    inner sep=0pt,
    rotate=#2}}
    
\setlist[enumerate]{leftmargin=20pt,itemsep=0pt,topsep=0pt}
\setlist[enumerate,1]{label=\emph{(\roman*)}}
\setlist[itemize]{leftmargin=20pt,itemsep=0pt,topsep=0pt}

\makeatletter
\renewcommand\section{\@startsection {section}{1}{\z@}%
                                   {-3.5ex \@plus -1ex \@minus -.2ex}%
                                   {1.3ex \@plus.2ex}%
                                   {\normalfont\bf\large}}
\makeatother

\makeatletter
\renewcommand\subsection{\@startsection {subsection}{1}{\z@}%
                                   {-3.5ex \@plus -1ex \@minus -.2ex}%
                                   {0.1ex \@plus.2ex}%
                                   {\normalfont\bf\normalsize}}
\makeatother

\newcommand{\discfarey}{
   \draw[gray,ultra thin] (0,-1) -- (0,1); 
 
   \foreach \thetaone/\thetatwo/\startangle/\finishangle/\R in 
{-90/0/-180/-270/1,-90.0/0.0/-180.0/-270.0/1.0,-90.0/-36.87/-180.0/-306.87/0.5,-90.0/-53.13/-180.0/-323.13/0.333,-90.0/-61.928/-180.0/-331.928/0.25,-90.0/-67.38/-180.0/-337.38/0.2,-90.0/-71.075/-180.0/-341.075/0.167,-90.0/-73.74/-180.0/-343.74/0.143,-90.0/-75.75/-180.0/-345.75/0.125,-90.0/-77.32/-180.0/-347.32/0.111,-90.0/-78.579/-180.0/-348.579/0.1,-90.0/-79.611/-180.0/-349.611/0.091,-90.0/-80.473/-180.0/-350.473/0.083,-90.0/-81.203/-180.0/-351.203/0.077,-90.0/-81.829/-180.0/-351.829/0.071,-90.0/-82.372/-180.0/-352.372/0.067,-90.0/-82.847/-180.0/-352.847/0.063,-90.0/-83.267/-180.0/-353.267/0.059,-90.0/-83.64/-180.0/-353.64/0.056,-90.0/-83.974/-180.0/-353.974/0.053,-90.0/-84.275/-180.0/-354.275/0.05,-84.275/-83.974/-174.275/-353.974/0.003,-83.974/-83.64/-173.974/-353.64/0.003,-83.64/-83.267/-173.64/-353.267/0.003,-83.267/-82.847/-173.267/-352.847/0.004,-82.847/-82.372/-172.847/-352.372/0.004,-82.372/-81.829/-172.372/-351.829/0.005,-81.829/-81.203/-171.829/-351.203/0.005,-81.203/-80.473/-171.203/-350.473/0.006,-80.473/-79.611/-170.473/-349.611/0.008,-79.611/-78.579/-169.611/-348.579/0.009,-78.579/-77.32/-168.579/-347.32/0.011,-77.32/-75.75/-167.32/-345.75/0.014,-75.75/-73.74/-165.75/-343.74/0.018,-73.74/-71.075/-163.74/-341.075/0.023,-71.075/-67.38/-161.075/-337.38/0.032,-67.38/-61.928/-157.38/-331.928/0.048,-61.928/-53.13/-151.928/-323.13/0.077,-61.928/-58.109/-151.928/-328.109/0.033,-58.109/-53.13/-148.109/-323.13/0.043,-53.13/-36.87/-143.13/-306.87/0.143,-53.13/-46.397/-143.13/-316.397/0.059,-53.13/-48.888/-143.13/-318.888/0.037,-48.888/-46.397/-138.888/-316.397/0.022,-46.397/-36.87/-136.397/-306.87/0.083,-46.397/-43.603/-136.397/-313.603/0.024,-43.603/-36.87/-133.603/-306.87/0.059,-43.603/-42.075/-133.603/-312.075/0.013,-42.075/-36.87/-132.075/-306.87/0.045,-42.075/-41.112/-132.075/-311.112/0.008,-41.112/-36.87/-131.112/-306.87/0.037,-36.87/0.0/-126.87/-270.0/0.333,-36.87/-22.62/-126.87/-292.62/0.125,-36.87/-28.072/-126.87/-298.072/0.077,-36.87/-30.51/-126.87/-300.51/0.056,-36.87/-31.891/-126.87/-301.891/0.043,-36.87/-32.779/-126.87/-302.779/0.036,-32.779/-31.891/-122.779/-301.891/0.008,-31.891/-30.51/-121.891/-300.51/0.012,-30.51/-28.072/-120.51/-298.072/0.021,-28.072/-22.62/-118.072/-292.62/0.048,-28.072/-25.989/-118.072/-295.989/0.018,-25.989/-22.62/-115.989/-292.62/0.029,-22.62/0.0/-112.62/-270.0/0.2,-22.62/-16.26/-112.62/-286.26/0.056,-22.62/-18.925/-112.62/-288.925/0.032,-18.925/-16.26/-108.925/-286.26/0.023,-16.26/0.0/-106.26/-270.0/0.143,-16.26/-12.68/-106.26/-282.68/0.031,-16.26/-14.25/-106.26/-284.25/0.018,-14.25/-12.68/-104.25/-282.68/0.014,-12.68/0.0/-102.68/-270.0/0.111,-12.68/-10.389/-102.68/-280.389/0.02,-10.389/0.0/-100.389/-270.0/0.091,-10.389/-8.797/-100.389/-278.797/0.014,-8.797/0.0/-98.797/-270.0/0.077,-8.797/-7.628/-98.797/-277.628/0.01,-7.628/0.0/-97.628/-270.0/0.067,-7.628/-6.733/-97.628/-276.733/0.008,-6.733/0.0/-96.733/-270.0/0.059,-6.733/-6.026/-96.733/-276.026/0.006,-6.026/0.0/-96.026/-270.0/0.053,-6.026/-5.453/-96.026/-275.453/0.005,-5.453/0.0/-95.453/-270.0/0.048,-5.453/-4.979/-95.453/-274.979/0.004,-4.979/0.0/-94.979/-270.0/0.043,-4.979/-4.581/-94.979/-274.581/0.003,-4.581/0.0/-94.581/-270.0/0.04,-4.581/-4.242/-94.581/-274.242/0.003,-4.242/0.0/-94.242/-270.0/0.037,-4.242/-3.95/-94.242/-273.95/0.003,-3.95/0.0/-93.95/-270.0/0.034,-3.95/-3.695/-93.95/-273.695/0.002,-3.695/0.0/-93.695/-270.0/0.032,-3.695/-3.471/-93.695/-273.471/0.002,-3.471/0.0/-93.471/-270.0/0.03,-3.471/-3.273/-93.471/-273.273/0.002,-3.273/0.0/-93.273/-270.0/0.029,-3.273/-3.096/-93.273/-273.096/0.002,-3.096/0.0/-93.096/-270.0/0.027,-3.096/-2.938/-93.096/-272.938/0.001,-2.938/0.0/-92.938/-270.0/0.026,-2.938/-2.794/-92.938/-272.794/0.001,-2.794/0.0/-92.794/-270.0/0.024}   
    {
      \tikzmath
      {
      	    \costheta = cos(\thetaone);
      	    \sintheta = sin(\thetaone);
	} 
	\draw[gray,ultra thin](\costheta,\sintheta) arc (\startangle:\finishangle:\R); 
     	\draw[gray,ultra thin](-\costheta,\sintheta) arc (180-\startangle:180-\finishangle:\R); 
	\draw[gray,ultra thin](\costheta,-\sintheta) arc (-\startangle:-\finishangle:\R); 
	\draw[gray,ultra thin](-\costheta,-\sintheta) arc (\startangle+180:\finishangle+180:\R); 
   }
   \draw (0,0) circle (1); 
}

\def\dischorocycle[#1](#2:#3)
{
	\tikzmath{
	    \A=#2;
      	    \B=#3;
      	    if  #3!=0 then {
      	       \thetaone = 2*atan(\A/\B))-90;
      	      \R = 1/(1+\A*\A+\B*\B);
          } else {
             \thetaone = 90;
             \R=0.5;
          };
	} 
	\draw[#1] (\thetaone:{1-\R}) circle (\R);
}    
    
\def\hgline[#1](#2:#3:#4:#5)
{
      \tikzmath{
      	    if  #3!=0 then {\thetaone = 2*atan(#2/#3))-90;} else {\thetaone = 90;}; 
      	    if  #5!=0 then {\thetatwo = 2*atan(#4/#5))-90;} else {\thetatwo = 90;}; 
          \anglephi = abs(\thetatwo-\thetaone);
          \startangle = \thetaone-90;
         \finishangle = \startangle+\anglephi-180;
         \R = abs((#2*#5-#3*#4)/(#2*#4+#3*#5));
	} 
	\draw[#1]($({cos(\thetaone)},{sin(\thetaone)})$) arc (\startangle:\finishangle:\R); 
     	\draw[#1]($({-cos(\thetaone)},{sin(\thetaone)})$) arc (180-\startangle:180-\finishangle:\R); 
	\draw[#1]($({cos(\thetaone)},{-sin(\thetaone)})$) arc (-\startangle:-\finishangle:\R); 
	\draw[#1]($({-cos(\thetaone)},{-sin(\thetaone)})$) arc (\startangle+180:\finishangle+180:\R); 
}

\def\shline[#1](#2:#3:#4:#5)
{
      \tikzmath{
      	    if  #3!=0 then {\thetaone = 2*atan(#2/#3))-90;} else {\thetaone = 90;}; 
      	    if  #5!=0 then {\thetatwo = 2*atan(#4/#5))-90;} else {\thetatwo = 90;}; 
          \anglephi = abs(\thetatwo-\thetaone)/2;
          if \thetaone<\thetatwo then {\thetamin = \thetaone;} else {\thetamin = \thetatwo;};
          \startangle = \thetamin-90;
         \finishangle = \startangle+2*\anglephi-180;
         \midangle = 90+0.5*(\startangle+\finishangle);
         if (\thetatwo==90) then {\midangle = \midangle -180;}; 
         \R = tan(\anglephi);
         \S= min(0.4*sqrt(abs(\R)),0.3);
         if \S < 0.1 then {\S =0;};
	} 
	\draw[#1]($({cos(\thetamin)},{sin(\thetamin)})$) arc (\startangle:\finishangle:\R) node[pos=0.5,pointer={\S}{\midangle}]{};
}

\def\fareylabel(#1:#2)
{
      \tikzmath{
      	    if  #2!=0 then {\thetaone = 2*atan(#1/#2))-90;} else {\thetaone = 90;}; 
  	} 
	\node[scale=0.85] at ($({1.1*cos(\thetaone)},{1.1*sin(\thetaone)})$) {$\tfrac{#1}{#2}$}; 
}

\AtEndDocument{%
	\par
	\medskip
	\begin{tabular}{@{}l@{}}%
		{Oleg Karpenkov  and Andrei Zabolotskii}\\
		{Department of Mathematical Sciences, University of Liverpool,}\\
		{Liverpool, L69 7ZL, United Kingdom}\\
		\textit{E-mail address}: \texttt{karpenk@liverpool.ac.uk, andrei.zabolotskii@liverpool.ac.uk}\\
		\\
		{Ian Short}\\
		{School of Mathematics and Statistics, The Open University,}\\
		{Milton Keynes, MK7 6AA, United Kingdom}\\
		\textit{E-mail address}: \texttt{ian.short@open.ac.uk}\\	
		\\
		{Matty van Son}\\
		{Department of Computer Science, Mathematics \& Physics, University of the West Indies,}\\
		{Cave Hill Campus, P.O. Box 64, Bridgetown BB11000, Barbados}\\
		\textit{E-mail address}: \texttt{matty.vanson@uwi.edu}
\end{tabular}}

\title{ \vspace{-6ex}\bf \large Classifying integer tilings and hypertilings\footnotetext{%
		\hspace{-7pt}2020 Mathematics Subject Classification: Primary 05E16, 15B36; Secondary 11B57.
		
		Key words: Bhargava cube, frieze, Farey graph, hyperdeterminant, integer hypertiling, integer tiling, $\text{SL}_2$-tiling. }}
		
\author{\normalsize Oleg Karpenkov, Ian Short, Matty van Son, and Andrei Zabolotskii}

\date{\vspace{-3ex}}

\begin{document}

\maketitle

\begin{abstract}
There are two objectives to this work: to classify all tame integer tilings and to classify all tame integer hypertilings. Motivation for the first objective comes from Conway and Coxeter's modelling of positive integer friezes using triangulated polygons, which has received significant attention since the discovery of cluster algebras by Fomin and Zelevinsky in 2002. Assem, Reutenauer, and Smith introduced $\text{SL}_2$-tilings as generalisations of friezes, and Bessenrodt, Holm, and J\o rgensen classified positive integer $\text{SL}_2$-tilings using infinite triangulated polygons. Here we consider $N$-tilings, of which $\text{SL}_2$-tilings are the case $N=1$. We provide a geometric model for all tame integer $N$-tilings using a generalisation of the Farey graph in the hyperbolic plane. Highlights of this model include classifications of all positive integer $N$-tilings and of all positive rational friezes, with entries encoded by lambda lengths or weight data of triangulated polygons.

The second objective is motivated by Bhargava's celebrated study of binary quadratic forms using integer cubes and by an observation of Demonet et al.\ that there is essentially only one three-dimensional positive integer tiling with $\text{SL}_2$ cross sections. We consider a richer class of three-dimensional tilings, which we call hypertilings, using the Cayley hyperdeterminant. We classify all tame integer hypertilings using generalised Farey graphs; remarkably, those with Cayley hyperdeterminant 1 prove to have a  simple description in terms of triple Hadamard products of integer pairs.
\end{abstract}

\section{Introduction}\label{section1}

This paper has a twofold objective: to classify all tame integer tilings and to classify all tame integer hypertilings. The provenance for such classifications is Conway and Coxeter's modelling of positive integer friezes using triangulated polygons \cite{CoCo1973}. Models of this type have blossomed since the discovery of cluster algebras by Fomin and Zelevinsky in the 2000s \cites{FoZe2002} and the work of Caldero and Chapoton \cite{CaCh2006}, who showed that Conway and Coxeter's friezes arise from cluster algebras of Dynkin type $A$. Later, Assem, Reutenauer, and Smith introduced $\text{SL}_2$-tilings \cite{AsReSm2010} (see also \cite{BeRe2010}), generalising friezes. 

The positivity conjecture of Fomin and Zelevinsky, which was verified for all skew-symmetric cluster algebras by Lee and Schiffler \cite{LeSc2015}, led to interest in positive integer $\text{SL}_2$-tilings and infinite friezes. Bessenrodt, Holm, and J\o rgensen offered a combinatorial classification of the former  \cite{BeHoJo2017} and Baur, Parsons, and Tschabold provided the same sort of classification for the latter \cite{BaPaTs2016}. The second author reinterpreted these classifications using the geometry of the Farey graph in \cite{Sh2023}, building on the vision of Morier-Genoud, Ovsienko, and Tabachnikov \cite{MoOvTa2015}. Then, in \cite{FeKaSeTu2023}, Felikson et al.\ took the same approach for a three-dimensional Farey graph, using Eisenstein integers, and the latter three authors of this paper extended this to more general rings in \cite{ShVaZa2025}.

Inspiration for this work comes from the elegant result of Demonet et al.\ \cite{DePlRuTu2018} that there is (essentially) a \emph{unique} three-dimensional tiling of positive integers with the property that every cross section is an $\text{SL}_2$-tiling. Here we consider broader classes of three-dimensional integer tilings in which every 2-by-2-by-2 block is required to have Cayley hyperdeterminant 1 or, more generally, $N$. Such cubes of integers are known as \emph{Bhargava cubes} following the pioneering work of Bhargava \cite{Bh2004} who interpreted Gauss composition of binary quadratic forms using these cubes. We combine Bhargava's approach with a more general concept of Farey graphs to model three-dimensional integer tilings, which we call \emph{hypertilings}. The cross sections of hypertilings are \emph{not} required to be $\text{SL}_2$-tilings -- they can be $N$-tilings, where each 2-by-2 square has determinant $N$ (and $N$ depends on the cross section). For this reason, we begin with a comprehensive study of $N$-tilings, which advances all previous classifications of integer $\text{SL}_2$-tilings. An appealing consequence of this is a complete description of all positive \emph{rational} friezes using triangulated polygons in which frieze entries can be specified geometrically by lambda lengths or encoded arithmetically by weight data of triangulated polygons. 

We will now introduce our main results, first for integer tilings and then for integer hypertilings. Throughout, $\mathcal{I}$, $\mathcal{J}$, and $\mathcal{K}$ are index sets of consecutive integers of length at least 3, possibly finite, possibly infinite, each of which contains $0$ and $1$. We write $\mathcal{I}'$ for the set obtained from $\mathcal{I}$ by removing the least element from $\mathcal{I}$, if it has one, and we write $\mathcal{I}^*$ for the set obtained from $\mathcal{I}'$ by removing the greatest element from $\mathcal{I}'$, if it has one.

\subsection{Integer tilings}

In this work, integer tilings are considered to be types of (possibly infinite) integer matrices. They are unrelated to tilings from the theory of tessellations. We define a \emph{subblock} of a matrix to be a block submatrix; that is, a submatrix with contiguous entries.

\begin{definition}
For any integer $N$, an \emph{$N$-tiling} is a function $\mathbf{M}\colon \mathcal{I}\times \mathcal{J}\longrightarrow \mathbb{Z}$, for some index sets $\mathcal{I}$ and $\mathcal{J}$, with the property that every 2-by-2 subblock has determinant $N$. An \emph{integer tiling} is an $N$-tiling, for some value of $N$. We write $m_{ij}$ for  $\mathbf{M}(i,j)$.
\end{definition}

In Section~\ref{section zero} we will consider 0-tilings. For now, we assume that $N$ is a nonzero integer. A 1-tiling is what is usually called an $\text{SL}_2$-tiling. Clearly, we can create an $N$-tiling from a 1-tiling by multiplying every other row, or every other column, by $N$. There are plenty of $N$-tilings that do not arise in this fashion.

\begin{definition}
An $N$-tiling is \emph{tame} if the determinant of each 3-by-3 subblock is 0. 
\end{definition}

This definition resembles that for tame $\text{SL}_2$-tilings. To get to grips with $N$-tilings, we need further data (which is trivial for $\text{SL}_2$-tilings) as follows.  

Consider a tame $N$-tiling $\mathbf{M}$. The \emph{greatest common divisor} $K$ of $\mathbf{M}$ is the greatest common divisor of all the entries from $\mathbf{M}$. Clearly $K=1$ for an $\text{SL}_2$-tiling. Now, suppose for the moment that $N$ is positive. Let $t$ be the greatest common divisor of all the 2-by-2 minors of $\mathbf{M}$, the Pl\"ucker coordinates of $\mathbf{M}$. Let $r$ be the greatest common divisor of the 2-by-2 minors from any two consecutive columns of $\mathbf{M}$, and let $R=N/r$. Likewise, let $s$ be the greatest common divisor of the 2-by-2 minors from any two consecutive rows of $\mathbf{M}$, and let $S=N/s$. We will see later that $R$ and $S$ are independent of which columns and rows are chosen. We will also see that $RS=N/t$. Notice that $t$ is divisible by $K^2$. Hence $N=K^2LRS$, where $L=t/K^2$.

When $N$ is negative, we replace $r$, $s$, and $t$ with their negatives and then define $R$, $S$, and $L$ as before. Consequently, $R$ and $S$ are (still) positive but now $L$ is negative.

\begin{definition}
For a tame $N$-tiling $\mathbf{M}$, the integers $K$, $L$, $R$, and $S$ are called the \emph{tameness parameters} of $\mathbf{M}$. We will usually write these as a 4-tuple $(K,L,R,S)$. We call $R$ the \emph{vertical tameness parameter} of $\mathbf{M}$ and $S$ the \emph{horizontal tameness parameter} of $\mathbf{M}$.
\end{definition}

For example, consider the tame 9-tiling $\mathbf{M}$ with greatest common divisor 1 shown in Figure~\ref{figure1}(a). The greatest common divisor of the  2-by-2 minors from the two shaded vertical columns is 9, so the vertical tameness parameter is $R=9/9=1$.  Similarly, the horizontal tameness parameter is $S=9/3=3$, so $L=3$. Hence the tameness parameters of $\mathbf{M}$ are $(1,3,1,3)$.

\begin{figure}[ht]
\hspace*{30pt}\begin{subfigure}[b]{0.5\textwidth}
	\begin{tikzpicture}
		\fill[rounded corners=2pt,fill=gammacol,fill opacity=0.3,xshift=-53pt] (0,-2.7) rectangle (1.4,2.7);
	\fill[rounded corners=2pt,fill=alphacol,fill opacity=0.3,yshift=-9pt] (-2.7,0) rectangle (2.7,1.4);
	\node at (0,0) {
		\(
    \mleft(\vcenter{
\xymatrix @-0.2pc @!0 {
67 & 144 & 29 & 30 & 11 & 36 & 13 \\ 46 & 99 & 20 & 21 & 8 & 27 & 10 \\ 25 & 54 & 11 & 12 & 5 & 18 & 7 \\ 4 & 9 & 2 & 3 & 2 & 9 & 4 \\ 7 & 18 & 5 & 12 & 11 & 54 & 25 \\ 10 & 27 & 8 & 21 & 20 & 99 & 46 \\ 13 & 36 & 11 & 30 & 29 & 144 & 67}}\mright)
\)
	};
	\end{tikzpicture}	
\caption{}	
\end{subfigure}	
\begin{subfigure}[b]{0.5\textwidth}
	\begin{tikzpicture}
	\node at (0,0) {
		\(
  \mleft(\vcenter{
  \xymatrix @-0.2pc @!0 {
\tfrac{67}{3} & 48 & \tfrac{29}{3} & 10 & \tfrac{11}{3} & 12 & \tfrac{13}{3} \\ \tfrac{46}{3} & 33 & \tfrac{20}{3} & 7 & \tfrac{8}{3} & 9 & \tfrac{10}{3} \\ \tfrac{25}{3} & 18 & \tfrac{11}{3} & 4 & \tfrac{5}{3} & 6 & \tfrac{7}{3} \\ \tfrac{4}{3} & 3 & \tfrac{2}{3} & 1 & \tfrac{2}{3} & 3 & \tfrac{4}{3} \\ \tfrac{7}{3} & 6 & \tfrac{5}{3} & 4 & \tfrac{11}{3} & 18 & \tfrac{25}{3} \\ \tfrac{10}{3} & 9 & \tfrac{8}{3} & 7 & \tfrac{20}{3} & 33 & \tfrac{46}{3} \\ \tfrac{13}{3} & 12 & \tfrac{11}{3} & 10 & \tfrac{29}{3} & 48 & \tfrac{67}{3}}}\mright)
 \)
	};
	\end{tikzpicture}	
\caption{}		
\end{subfigure}
\caption{A tame 9-tiling over $\mathbb{Z}$ (left) and the corresponding tame $\text{SL}_2$-tiling over $\frac13\mathbb{Z}$ (right)}
\label{figure1}
\end{figure}
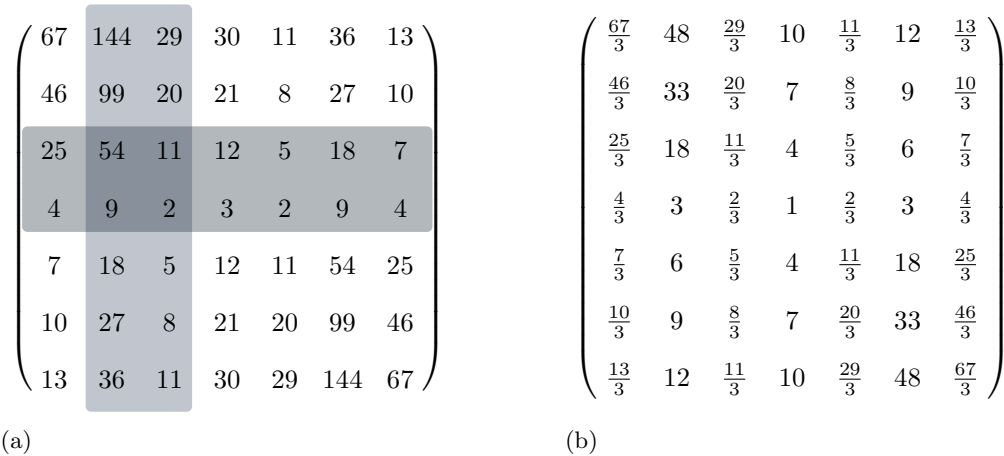

Figure~\ref{figure1}(b) is the tame $\text{SL}_2$-tiling over $\frac13\mathbb{Z}$ obtained from the tame 9-tiling of Figure~\ref{figure1}(a) by dividing each entry of that 9-tiling by 3. Through similar means we can obtain a one-to-one correspondence between tame $N^2$-tilings over $\mathbb{Z}$ and tame $\text{SL}_2$-tilings over $\frac{1}{N}\mathbb{Z}$.

To model integer tilings we will use a generalisation of the classical Farey graph, shown in Figure~\ref{figure2}. This is the graph with vertices comprising reduced rationals and a point at infinity and with an edge between two reduced rationals $a/b$ and $c/d$ if $ad-bc=\pm1$. The edges are represented by hyperbolic geodesics in the upper half-plane model of two-dimensional hyperbolic space. Sometimes we will switch to the Poincar\'e disc model of hyperbolic space instead.

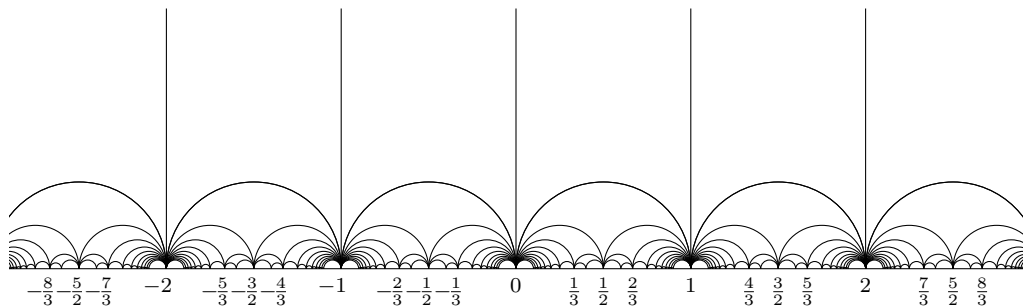
\begin{figure}[ht]
\centering
\begin{tikzpicture}[scale=2.3]
   \clip (-2.9,-0.3) rectangle (2.9,1.5);
   	
   \draw (-3,0) -- (3,0);	

   \foreach \leftend/\rightend in {0/1,0/0.5,0/0.333,0.333/0.5,0/0.25,0.25/0.333,0.333/0.4,0.4/0.5,0/0.2,0.2/0.25,0/0.166,0.166/0.2,0/0.143,0.143/0.166,0/0.125,0.125/0.143,0/0.111,0.111/0.125,0/0.1,0.1/0.111}{
    	\foreach \n in {-3,-2,-1,0,1,2} {
	    	\draw[] ({\n+\rightend},0) arc (0:180:{0.5*(\rightend-\leftend)}); 
	    	\draw[] ({-\n-\leftend},0) arc (0:180:{0.5*(\rightend-\leftend)});  
	 } 	
   }
   
    \foreach \bottom in {-3,-2,-1,0,1,2,3} {
    	\draw (\bottom,0) -- (\bottom,1.5);
    }
   
\foreach \pos/\label in {-2/$-2\phantom{-}$,-1/$-1\phantom{-}$,0/$0$,1/$1$,2/$2$,
-2.5/$-\!\tfrac52\phantom{-}$,-1.5/$-\!\tfrac32\phantom{-}$,-0.5/$-\!\tfrac12\phantom{-}$,0.5/$\tfrac12$,1.5/$\tfrac32$,2.5/$\tfrac52$,-2.667/$-\!\tfrac83\phantom{-}$,-2.333/$-\!\tfrac73\phantom{-}$,-1.667/$-\!\tfrac53\phantom{-}$,-1.333/$-\!\tfrac43\phantom{-}$,-0.667/$-\!\tfrac23\phantom{-}$,-0.333/$-\!\tfrac13\phantom{-}$,0.333/$\tfrac13$,0.667/$\tfrac23$,1.333/$\tfrac43$,1.667/$\tfrac53$,2.333/$\tfrac73$,2.667/$\tfrac83$}{
    	\node [below] at (\pos,0) {{\footnotesize \label}}; 	
   }

\end{tikzpicture}
\caption{The classical Farey graph}
\label{figure2}
\end{figure}

More generally, we define a class of Farey graphs indexed by the positive integers as follows.

\begin{definition}
For any positive integer $R$, we denote by $\mathscr{F}_R$ the directed graph with 
\begin{itemize}
\item vertices integer pairs $(a,b)\in\mathbb{Z}\times\mathbb{Z}$ for which $\gcd(a,b)$ is a factor of $R$,
\item a directed edge from $(a,b)$ to $(c,d)$ if $ad-bc=R$.
\end{itemize}
\end{definition}

We will often denote the vertex $(a,b)$ by the formal fraction $a/b$ because this fractional notation is suggestive for how we manipulate vertices of Farey graphs. All graphs $\mathscr{F}_R$ are connected. The notation $\mathscr{F}_R$ was used for a different class of graphs in \cite{ShVaZa2025}, where $R$ was a ring rather than an integer.

The directed graph $\mathscr{F}_1$ (which we also denote by just $\mathscr{F}$) is a double cover of the classical Farey graph when each classical Farey edge is considered to have two opposite arrows. More generally, to appreciate the geometry of $\mathscr{F}_R$ it helps to work with the quotient of $\mathscr{F}_R$ obtained by identifying pairs of vertices $a/b$ and $(-a)/(-b)$. A means for visualising this quotient is to use horocycles. To see this, observe that by restricting the usual (left) action of $\text{SL}_2(\mathbb{Z})$ on $\mathbb{Z}\times\mathbb{Z}$ we obtain an action on $\mathscr{F}_R$. Under this action, the orbits of vertices of $\mathscr{F}_R$ correspond to the positive divisors of $R$ (so the action is not transitive unless $R=1$).

For $(a,b)\in\mathbb{Z}\times\mathbb{Z}\setminus\{(0,0)\}$, let $H(a,b)$ denote the horocycle centred at $a/b$ of Euclidean diameter $1/b^2$, for $b\neq 0$ (and $H(a,0)$ is the horocycle centred at $\infty$ of height $a^2$). Evidently $H(-a,-b)=H(a,b)$. The group $\text{SL}_2(\mathbb{Z})$ acts on the collection of horocycles by M\"obius transformations in the usual way, and the map $H$ is equivariant with respect to these two actions of $\text{SL}_2(\mathbb{Z})$, on $\mathbb{Z}\times\mathbb{Z}\setminus\{(0,0)\}$ and on horocycles. There is an eloquent account of these actions in \cite{Sp2017}*{Section~5}.

It follows, then, that we can obtain a geometric model of $\mathscr{F}_R$ in the hyperbolic plane that respects the action of $\text{SL}_2(\mathbb{Z})$ by representing the vertex $a/b$ by the horocycle $H(a,b)$. See Figure~\ref{figure3}, for example, which illustrates the Farey graph $\mathscr{F}_2$ (with vertices $a/b$ and $(-a)/(-b)$ identified). There are two classes of horocycles: those with a lighter shade correspond to vertices $a/b$ with $\gcd(a,b)=1$ and those with a darker shade satisfy $\gcd(a,b)=2$.

\begin{figure}[ht]
\centering
\begin{tikzpicture}[scale=2.3]
   \clip (-2.9,-0.1) rectangle (2.9,2.4);
   	
   \draw (-3,0) -- (3,0);	

   \foreach \leftend/\rightend in {0/1,0/0.5,0/0.333,0.333/0.5,0/0.25,0.25/0.333,0.333/0.4,0.4/0.5,0/0.2,0.2/0.25,0/0.166,0.166/0.2,0/0.143,0.143/0.166,0/0.125,0.125/0.143,0/0.111,0.111/0.125,0/0.1,0.1/0.111}{
    	\foreach \n in {-3,-2,-1,0,1,2} {
	    	\draw[grey] ({\n+\rightend},0) arc (0:180:{0.5*(\rightend-\leftend)}); 
	    	\draw[grey] ({-\n-\leftend},0) arc (0:180:{0.5*(\rightend-\leftend)});  
	 } 	
   }
        
    \foreach \bottom in {-3,-2,-1,0,1,2,3} {
    	\draw[grey] (\bottom,0) -- (\bottom,2.5);
    }
   
\foreach \num/\denom/\label in {-3/1/$-\frac31\phantom{-}$,-2/1/$-\frac21\phantom{-}$,-1/1/$-\frac11\phantom{-}$,0/1/$\frac01$,1/1/$\frac11$,2/1/$\frac21$,3/1/$\frac31$,
-5/2/$-\frac52\phantom{-}$,-3/2/$-\frac32\phantom{-}$,-1/2/$-\frac12$,1/2/$\frac12$,3/2/$\frac32$,5/2/$\frac52$,1/3/$\frac13$,2/3/$\frac23$,4/3/$\frac43$,5/3/$\frac53$,7/3/$\frac73$,8/3/$\frac83$,-1/3/$-\frac13\phantom{-}$,-2/3/$-\frac23\phantom{-}$,-4/3/$-\frac43\phantom{-}$,-5/3/$-\frac53\phantom{-}$,-7/3/$-\frac73\phantom{-}$,-8/3/$-\frac83\phantom{-}$}{
   	\tikzmath{\radius = 1/(2*\denom*\denom); } 
   	\draw[fill=alphacol,fill opacity=0.2] ({\num/\denom},\radius) circle (\radius);
   	\ifnum \denom <4 
   		\node[scale=1.8*sqrt(\radius)] at ({\num/\denom},\radius) {\label};
	\fi
   }
   
\foreach \num/\denom/\label in {-6/2/$-\frac62\phantom{-}$,-4/2/$-\frac42\phantom{-}$,-2/2/$-\frac22\phantom{-}$,0/2/$\frac02$,2/2/$\frac22$,4/2/$\frac42$,6/2/$\frac62$,
-10/4/$-\frac{10}{4}\phantom{-}$,-6/4/$-\frac64\phantom{-}$,-2/4/$-\frac24$,2/4/$\frac24$,6/4/$\frac64$,10/4/$\frac{10}{4}$}{
   	\tikzmath{\radius = 1/(2*\denom*\denom); } 
   	\draw[fill=alphacol,fill opacity=0.5] ({\num/\denom},\radius) circle (\radius);
   	\ifnum \denom <4 
	   	\node[scale=1.8*sqrt(\radius)] at ({\num/\denom},\radius) {\label};
	\fi
   }
   
   \draw (-3,1) -- (3,1);
   \fill[alphacol,opacity=0.2] (-3,1) rectangle (3,3);
   \draw (-3,2) -- (3,2);
   \fill[alphacol,opacity=0.5] (-3,2) rectangle (3,3);
   
   \node[scale=1.5] at (0,1.5){$\frac10$};
   \node[scale=1.5] at (0,2.2){$\frac20$};	
\end{tikzpicture}
\caption{The Farey graph $\mathscr{F}_2$}
\label{figure3}
\end{figure}
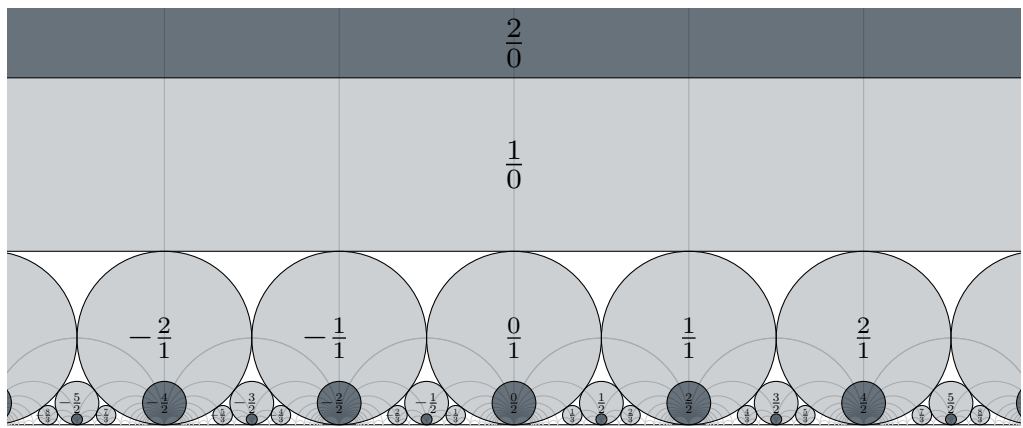

Edges in $\mathscr{F}_R$ can be interpreted either geometrically or combinatorially, as follows. For the geometric interpretation, two horocycles $a/b$ and $c/d$ are adjacent if and only if the signed hyperbolic distance $d$ between them is $2\log R$ (see \cite{Sp2017}*{Proposition~6.2}). Following Penner \cite{Pe1987}, this quantity $e^{d/2}$ is called the lambda length between the horocycles. For the combinatorial interpretation, we recall that the \emph{Farey distance} between the reduced rationals $p/q$ and $r/s$ in the \emph{classical} Farey graph is  $\Delta(p/q,r/s)=|ps-qr|$ (see \cite{MoOvTa2015}). Farey distance can be defined on any finite triangulated polygon by embedding that polygon in the Farey tessellation, and it can also be defined combinatorially by the process sometimes called Conway--Coxeter counting (see, for example, \cite{BeHoJo2017}). Now, given two vertices $a/b$ and $c/d$ of $\mathscr{F}_R$ there are positive integer weights $u$ and $v$ with $(a,b)=u(p,q)$ and $(c,d)=v(r,s)$, for reduced rationals $p/q$ and $r/s$. Then, there is an edge in $\mathscr{F}_R$ between $a/b$ and $c/d$ if and only if the \emph{weighted} Farey distance $uv\Delta(p/q,r/s)$ equals $R$. We will use this perspective shortly to classify positive rational friezes.

\begin{definition}
A \emph{path} in $\mathscr{F}_R$ is a sequence of vertices $(a_i/b_i)_{i\in \mathcal{I}}$ in $\mathscr{F}_R$, for some index set $\mathcal{I}$, for which there is a directed edge from $a_{i-1}/b_{i-1}$ to $a_i/b_i$, for each $i\in \mathcal{I}'$.  The path is \emph{minimal} if the greatest common divisor of the integers $a_jb_i-b_ja_i$, for $i,j\in \mathcal{I}$,  is $1$.
\end{definition}

To ease notation, we sometimes write $a_i/b_i$ for the path $(a_i/b_i)_{i\in \mathcal{I}}$. Notice that all paths in $\mathscr{F}_1$ are minimal because $a_jb_i-b_ja_i=1$ whenever $j=i-1$.

Now, let $L$ be a nonzero integer. We define
\[
\Gamma^0(L) = \left\{ \begin{pmatrix}a&bL\\ c& d\end{pmatrix} \in \text{SL}_2(\mathbb{Z}) : a,b,c,d\in\mathbb{Z}\right\}.
\]
This group is sometimes called the \emph{Hecke congruence subgroup} of level $L$. Observe that 
\[
\begin{pmatrix}a&b\\ cL& d\end{pmatrix}=\begin{pmatrix}1/L&0\\ 0& 1\end{pmatrix}\begin{pmatrix}a&bL\\ c& d\end{pmatrix}\begin{pmatrix}L&0\\ 0& 1\end{pmatrix}.
\]
We can define a left action of $\Gamma^0(L)$ on $\mathbb{Z}^2 \times \mathbb{Z}^2$ by
\[
\begin{pmatrix}a&bL\\ c& d\end{pmatrix} \mleft(\begin{pmatrix}x_1\\ x_2\end{pmatrix}, \begin{pmatrix}y_1\\ y_2\end{pmatrix} \mright) = \mleft( \begin{pmatrix}a&bL\\ c& d\end{pmatrix}\begin{pmatrix}x_1\\
x_2\end{pmatrix}, \begin{pmatrix}a&b\\ cL& d\end{pmatrix}\begin{pmatrix}y_1\\ y_2\end{pmatrix} \mright).
\]
This restricts to an action on $\mathscr{F}_R\times \mathscr{F}_S$ which preserves the collection of pairs of minimal paths in $\mathscr{F}_R\times \mathscr{F}_S$. Using this action we obtain our first main result, a one-to-one correspondence between pairs of paths and $N$-tilings.

\begin{maintheorem}{A}\label{theoremA}
Let $K$, $L$, $R$, $S$, $N$ be nonzero integers with $K,R,S>0$ and $N=K^2LRS$. The map
\[
\hspace*{-30pt}\Gamma^0(L)\Big\backslash\mleft\{\parbox{2.1cm}{\centering\textnormal{minimal paths in $\mathscr{F}_R$}}\mright\}\times\mleft\{\parbox{2.1cm}{\centering\textnormal{ minimal paths in
$\mathscr{F}_S$}}\mright\}\quad \longrightarrow\quad \mleft\{\parbox{4.6cm}{\centering\textnormal{tame~$N$-tilings with tameness parameters $(K,L,R,S)$}}\mright\}
\]
determined by 
\[
m_{ij} = K(a_id_j-Lb_ic_j),\qquad\text{$(i,j)\in\mathcal{I}\times\mathcal{J}$},
\]
for minimal paths $a_i/b_i$ and $c_j/d_j$ in $\mathscr{F}_R$ and $\mathscr{F}_S$, is a one-to-one correspondence.
\end{maintheorem}

Theorem~\ref{theoremA} generalises \cite{Sh2023}*{Theorem~1.1} which is the special case $N=1$. In that case $K$, $L$, $R$, and $S$ all equal 1 and the theorem simplifies dramatically. Theorem~\ref{theoremA} will be proved in Section~\ref{section83}.

To illustrate Theorem~\ref{theoremA}, consider the  minimal paths 
\[
\gamma=\langle \tfrac{10}{-3},\tfrac{7}{-2},\tfrac{4}{-1},\tfrac{1}{0},\tfrac{4}{1},\tfrac{7}{2},\tfrac{10}{3}\rangle\in\mathscr{F}_1\quad\text{and}\quad\delta=\langle \tfrac{3}{4},\tfrac{6}{9},\tfrac{1}{2},\tfrac{0}{3},\tfrac{-1}{2},\tfrac{-6}{9},\tfrac{-3}{4}\rangle\in\mathscr{F}_3.
\]
 These two paths are exhibited on the classical Farey graph in Figure~\ref{figure4} embedded in the unit disc. With $K=1$, $L=3$, $R=1$, and $S=3$, the formula $m_{ij}=a_id_j-3b_ic_j$ gives the tame 9-tiling with tameness parameters $(1,3,1,3)$ shown in Figure~\ref{figure1}(a) (where $a_i/b_i$ are the vertices of $\gamma$ and $c_j/d_j$ are the vertices of $\delta$).

\begin{figure}[H]
\centering
\begin{tikzpicture}[scale=2.3]
\discfarey

\foreach \numA/\denA/\numB/\denB in {4/-1/1/0,4/-1/7/-2,7/-2/10/-3,1/0/4/1,4/1/7/2,7/2/10/3}{
	   \shline[gammacol,ultra thick](\numA:\denA:\numB:\denB);   
}

\foreach \numA/\denA/\numB/\denB in {3/4/6/9,6/9/1/2,1/2/0/3,0/3/-1/2,-1/2/-6/9,-6/9/-3/4}{
	   \shline[deltacol,ultra thick](\numA:\denA:\numB:\denB);   
}

\foreach \num/\den in {4/-1,1/0,4/1}{
	   \fareylabel(\num:\den); 
}
\foreach \num/\den in {6/9,1/2,0/3,-1/2,-6/9}{
	   \fareylabel(\num:\den); 
}

\node[deltacol] at (250:0.5)  {$\delta$};
\node[gammacol] at (70:0.7) {$\gamma$};
\end{tikzpicture}
\caption{Paths $\gamma=\langle \tfrac{10}{-3},\tfrac{7}{-2},\tfrac{4}{-1},\tfrac{1}{0},\tfrac{4}{1},\tfrac{7}{2},\tfrac{10}{3}\rangle\in\mathscr{F}_1$ and $\delta=\langle \tfrac{3}{4},\tfrac{6}{9},\tfrac{1}{2},\tfrac{0}{3},\tfrac{-1}{2},\tfrac{-6}{9},\tfrac{-3}{4}\rangle\in\mathscr{F}_3$}
\label{figure4}
\end{figure}
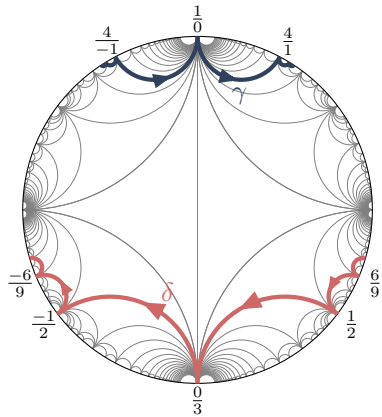

The configuration of two clockwise paths in Figure~\ref{figure4} is representative of the geometry of pairs of paths more generally that determine \emph{positive} integer tilings. In Section~\ref{section14} we provide a complete characterisation of all positive integer tilings in terms of pairs of paths, as a corollary to Theorem~\ref{theoremA}. This result is in the same spirit as \cite{Sh2023}*{Theorem~1.2} but more complex for its generality.

Another consequence of Theorem~\ref{theoremA} is that it allows us to completely characterise positive \emph{rational} friezes. An example of such a frieze is shown in Figure~\ref{figure5}; we give a formal definition in Section~\ref{section77}. 

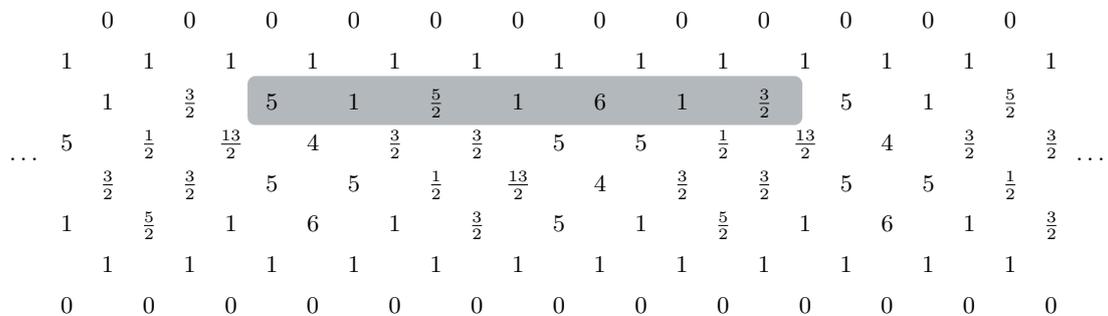
\begin{figure}[ht]
	\centering
	\begin{tikzpicture}
		\fill[rounded corners=3pt,fill=alphacol,fill opacity=0.3] (-4.1,0.5) rectangle (3.2,1.15);
	\node at (0,0) {\small
		\(
		\vcenter{
  \xymatrix @-0.72pc @!0 {
    && 0 && 0 && 0 && 0 && 0 && 0 && 0 && 0 && 0 && 0 && 0 && 0 &&\\
        &1 && 1 && 1 && 1 && 1 && 1 && 1 && 1 && 1 && 1 && 1 && 1 && 1 &\\
    && 1 && \tfrac32 && 5 && 1 && \tfrac52 && 1 && 6 && 1 && \tfrac32 && 5 && 1 && \tfrac52 &&\\
 \raisebox{-4pt}{$\dotsc$}         &5 && \tfrac12 && \tfrac{13}2 && 4 && \tfrac32 && \tfrac32 && 5 && 5 &&  \tfrac12 && \tfrac{13}2 && 4 && \tfrac32 && \tfrac32 &\raisebox{-4pt}{$\dotsc$}  \\
    && \tfrac32 && \tfrac32 && 5 && 5 && \tfrac12 && \tfrac{13}2 && 4 && \tfrac32 && \tfrac32 && 5 && 5 && \tfrac12 &&\\
        &1 && \tfrac52 && 1 && 6 && 1 && \tfrac32 && 5 && 1 && \tfrac52 && 1 && 6 && 1 && \tfrac32 &\\
    && 1 && 1 && 1 && 1 && 1 && 1 && 1 && 1 && 1 && 1 && 1 && 1 && \\
       & 0 && 0 && 0 && 0 && 0 && 0 && 0 && 0 && 0 && 0 && 0 && 0 && 0 &\\
                       }
          }
          		\)
	};
\end{tikzpicture}
\caption{A positive rational frieze}
\label{figure5}
\end{figure}

It is well known that all positive rational (or real) friezes are periodic, so in fact all entries of a positive rational frieze belong to $\tfrac{1}{N}\mathbb{Z}$, for some positive integer $N$. Let us define a \emph{closed clockwise path} in $\mathscr{F}_R$ to be a finite path with terminal pair equal to the negative of the initial pair and whose projection traverses the extended real line (endowed with cyclic order) once clockwise. We then have the following classification of  positive friezes over $\tfrac{1}{N}\mathbb{Z}$, which is proved in Section~\ref{section77}, where we also explain finer aspects of the theorem such as the index sets.

\begin{maintheorem}{B}\label{theoremB}
Let $N$, $K$, and $R$ be positive integers with $N=KR$. The map
\[
\hspace*{-30pt}\textnormal{SL}_2(\mathbb{Z})\Big\backslash\mleft\{\parbox{4.0cm}{\centering\textnormal{minimal closed clockwise paths in $\mathscr{F}_R$ of length~$n$}}\mright\}\quad \longrightarrow\quad \mleft\{\parbox{5.5cm}{\centering\textnormal{positive friezes over $\tfrac{1}{N}\mathbb{Z}$ of width $n$ and greatest common divisor $K$}}\mright\}
\]
determined by 
\[
m_{ij} = \frac{1}{R}(a_jb_i-b_ja_i), \quad 0\leq i-j\leq n,
\]
for a minimal closed clockwise path $a_i/b_i$ in $\mathscr{F}_R$, is a one-to-one correspondence.
\end{maintheorem}

To illustrate Theorem~\ref{theoremB}, consider the minimal closed clockwise path 
\[
\gamma=\langle \tfrac10,\tfrac32,\tfrac22,\tfrac01,\tfrac{-2}3,\tfrac{-2}2,\tfrac{-3}2,\tfrac{-1}0\rangle
\]
shown in Figure~\ref{figure46}(a). The corresponding positive frieze over $\tfrac12\mathbb{Z}$ with entries $\tfrac12(a_jb_i-b_ja_i)$ is that of Figure~\ref{figure5}. In Figure~\ref{figure46}(b) the horocycles from the horocycle representation of $\mathscr{F}_2$  are shown at the vertices of $\gamma$. Let $\lambda_{ij}$ denote the lambda length between the $i$th and $j$th horocycles; that is, $\lambda_{ij}=e^{d_{ij}/2}$, where $d_{ij}$ is the signed hyperbolic distance between the two horocycles. Then the $(i,j)$th entry of the positive rational frieze of Figure~\ref{figure5} is equal to $\tfrac12\lambda_{ij}$. 

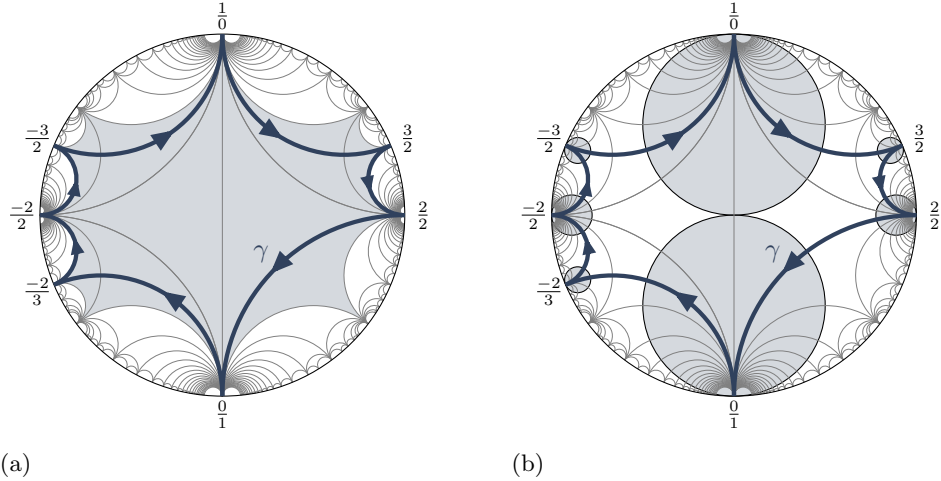
\begin{figure}[ht]
\centering
\begin{subfigure}[b]{0.45\textwidth}
\begin{tikzpicture}[scale=2.4]
\draw[draw=none,fill=gammacol,opacity=0.2] (0,1) arc (-360:-486.87:0.5) arc (-306.87:-472.62:0.125) arc (-292.62:-450.0:0.2) arc (-270.0:-427.38:0.2) arc (-247.38:-413.13:0.125) arc (-233.13:-360.0:0.5) arc (-180.0:-306.87:0.5) arc (-126.87:-270.0:0.333) arc (-90.0:-247.38:0.2) arc (-67.38:-233.13:0.125) arc (-53.13:-180.0:0.5);

\discfarey

\foreach \numA/\denA/\numB/\denB in {1/0/3/2,3/2/2/2,2/2/0/1,0/1/-2/3,-2/3/-2/2,-2/2/-3/2,-3/2/-1/0}{
	   \shline[gammacol,ultra thick](\numA:\denA:\numB:\denB);   
}

\foreach \num/\den in {1/0,3/2,2/2,0/1,-2/3,-2/2,-3/2}{
	   \fareylabel(\num:\den); 
}
\node[gammacol] at (-45:0.3) {$\gamma$};
\end{tikzpicture}
\caption{}
\end{subfigure}
\begin{subfigure}[b]{0.45\textwidth}
\begin{tikzpicture}[scale=2.4]
\foreach \num/\den in {1/0,3/2,2/2,0/1,-2/3,-2/2,-3/2}{
	   \dischorocycle[black,fill=gammacol,fill opacity=0.2](\num:\den);
}

\discfarey

\foreach \numA/\denA/\numB/\denB in {1/0/3/2,3/2/2/2,2/2/0/1,0/1/-2/3,-2/3/-2/2,-2/2/-3/2,-3/2/-1/0}{
	   \shline[gammacol,ultra thick](\numA:\denA:\numB:\denB);   
}

\foreach \num/\den in {1/0,3/2,2/2,0/1,-2/3,-2/2,-3/2}{
	   \fareylabel(\num:\den); 
}
\node[gammacol] at (-45:0.3) {$\gamma$};

\end{tikzpicture}
\caption{}
\end{subfigure}
\caption{The path $\gamma$ in a triangulated polygon (left) and with horocycles (right)}
\label{figure46}
\end{figure}

Any positive rational or real frieze is determined by the periodic part of its third row; this sequence, up to cyclic permutations, is known as the \emph{quiddity cycle} of the frieze. The quiddity cycle is highlighted in Figure~\ref{figure5}. We will explain how to classify positive rational quiddity cycles in the manner of Conway and Coxeter's classification of positive integral quiddity cycles. To achieve this, we use weighted triangulated polygons, which are defined as follows. 

Consider a selection $v_0,v_1,\dots,v_{m-1}$ of the vertices of a triangulated polygon, listed in clockwise order. Let $d_i$ be the Farey distance between $v_{i-1}$ and $v_i$ in the triangulated polygon (with $v_{m}=v_0$). To each vertex $v_i$ we assign a positive integer weight $w_i$ (with $w_m=w_0$) such that the products $d_iw_{i-1}w_i$ are all equal to the same positive integer $N$, for $i=1,2,\dots,m$. A triangulated polygon with a collection of vertices of this type is called a \emph{weighted triangulated polygon}.

For example, consider the triangulated decagon of Figure~\ref{figure:decagon}(a) -- ignore the integer labels for now. Seven of the ten vertices are coloured black or white; these are the vertices $v_0,v_1,\dots,v_6$. We define the black ones to have weight 1 and the white ones to have weight 2. This triangulated decagon with these weighted vertices is a weighted triangulated polygon with $N=2$. It is a Euclidean version of the combinatorially equivalent hyperbolic ideal triangulated polygon shaded in Figure~\ref{figure46}(a). The black and white vertices correspond to the vertices of the path shown in Figure~\ref{figure46}(a), with black for vertices $a/b$ with $\gcd(a,b)=1$ and white for vertices $a/b$ with $\gcd(a,b)=2$.

The integer label at the distinguished vertex $v_i$ of the weighted triangulated polygon is the product of the weights $w_{i-1}$ and $w_{i+1}$ of the two neighbouring vertices and the Farey distance $\Delta(v_{i-1},v_{i+1})$ between those two vertices. Reading these labels clockwise in order we obtain the quiddity cycle of a positive integer 4-frieze, as shown in Figure~\ref{figure:decagon}(b). That is, we obtain a type of frieze where the diamonds have determinant 4 rather than~1 (this is an example of a frieze pattern with coefficients \cite{CuHoJo2020}, which are related to cluster algebras with coefficients \cite{FoZe2007}). On dividing each entry of this frieze by 2 we obtain the positive rational frieze of Figure~\ref{figure5}. 

This procedure constitutes a combinatorial interpretation of the one-to-one correspondence of Theorem~\ref{theoremB}, which shows that \emph{every} positive rational frieze arises in this way.

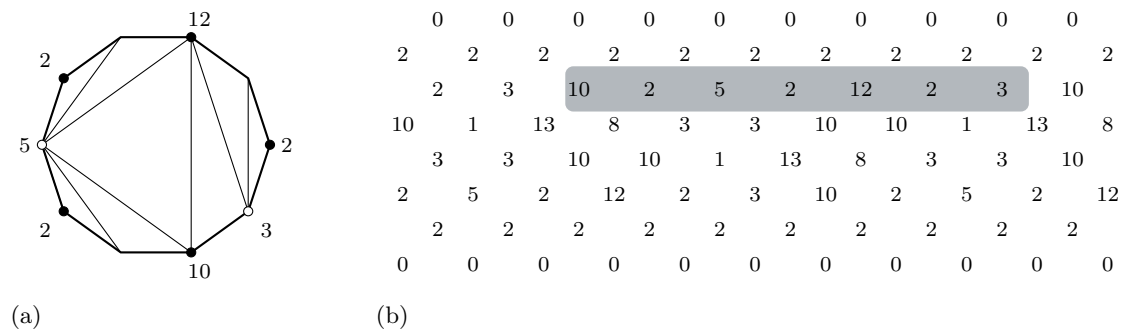
\begin{figure}[ht]
\begin{subfigure}[b]{0.25\textwidth}
\begin{tikzpicture}[scale=1.5]

\foreach \i [count=\j from 0] in {0,36,...,324} {
  \coordinate (P\j) at ({cos(\i)},{-sin(\i)});
}

\draw[thick] (P0) -- (P1) -- (P2) -- (P3) -- (P4) -- (P5) -- (P6) -- (P7) -- (P8) -- (P9) -- cycle;

\draw (P9) -- (P1) -- (P8) -- (P2) -- (P5) -- (P4);
\draw (P8) -- (P5) -- (P7);
\draw (P5) -- (P3);

\tikzmath{\rad = 0.04;} 
\filldraw[fill=black] (P0) circle (\rad) node[label={[label distance=-1mm]0:{\footnotesize 2}}] {};
\filldraw[fill=white] (P1) circle (\rad) node[label={[label distance=-1mm]-36:{\footnotesize 3}}] {};
\filldraw[fill=black] (P2) circle (\rad) node[label={[label distance=-1mm,xshift=-5pt]-72:{\footnotesize 10}}] {};
\filldraw[fill=black] (P4) circle (\rad) node[label={[label distance=-1mm]-144:{\footnotesize 2}}] {};
\filldraw[fill=white] (P5) circle (\rad) node[label={[label distance=-1mm]-180:{\footnotesize 5}}] {};
\filldraw[fill=black] (P6) circle (\rad) node[label={[label distance=-1mm]-216:{\footnotesize 2}}] {};
\filldraw[fill=black] (P8) circle (\rad) node[label={[label distance=-1mm,xshift=-5pt]-288:{\footnotesize 12}}] {};
\end{tikzpicture}
\caption{}
\end{subfigure}
\qquad\quad
\begin{subfigure}[b]{0.7\textwidth}\footnotesize
	\begin{tikzpicture}
		\fill[rounded corners=3pt,fill=alphacol,fill opacity=0.3] (-2.5,0.4) rectangle (3.6,1.0);
	\node at (0,0) {
		\(
\vcenter{
  \xymatrix @-0.9pc @!0 {
    & 0 && 0 && 0 && 0 && 0 && 0 && 0 && 0 && 0 && 0 &\\
        2 && 2 && 2 && 2 && 2 && 2 && 2 && 2 && 2 && 2 && 2 \\
    & 2 && 3 && 10 && 2 && 5 && 2 && 12 && 2 && 3 && 10 &\\
  10 && 1 && 13 && 8 && 3 && 3 && 10 && 10 && 1 && 13 && 8   \\
    & 3 && 3 && 10 && 10 && 1 && 13 && 8 && 3 && 3 && 10 &\\
        2 && 5 && 2 && 12 && 2 && 3 && 10 && 2 && 5 && 2 && 12 \\
    & 2 && 2 && 2 && 2 && 2 && 2 && 2 && 2 && 2 && 2 & \\
        0 && 0 && 0 && 0 && 0 && 0 && 0 && 0 && 0 && 0 && 0 \\
                       }
          }
\)
};
\end{tikzpicture}
\caption{}
\end{subfigure}
\caption{A weighted triangulated decagon (left) and the corresponding 4-frieze (right)}
\label{figure:decagon}
\end{figure}

\subsection{Integer hypertilings}

We now turn our attention to integer hypertilings. We will demonstrate that each tame integer hypertiling can be constructed from three paths in Farey graphs and a \emph{Bhargava cube}, which is an element of $\mathbb{Z}^2\otimes\mathbb{Z}^2\otimes\mathbb{Z}^2$, named in honour of the pioneering work of Bhargava \cite{Bh2004}. We can regard a Bhargava cube $M$ as a function $M\colon\{0,1\}^3\longrightarrow \mathbb{Z}$. Let us write $m_{ijk}$ for $M(i,j,k)$, and then we can represent the cube in the form
\begin{equation}\label{equation1}
\mleft(\begin{matrix}
m_{000} & m_{010} \\
m_{100} & m_{110}
\end{matrix}\
\middle\vert\
\begin{matrix}
m_{001} & m_{011} \\
m_{101} & m_{111}
\end{matrix}\mright).
\end{equation}
Here we can think of the left-hand square as lying above the right-hand square to form a cube of eight integers (and we use similar notation for larger cubes). There is a left action of the direct product $\text{SL}_2(\mathbb{Z})^3=\text{SL}_2(\mathbb{Z})\times\text{SL}_2(\mathbb{Z})\times\text{SL}_2(\mathbb{Z})$ on Bhargava cubes under which the triple $(A,B,C)\in\text{SL}_2(\mathbb{Z})^3$ sends the Bhargava cube $M$ to $M'$, where
\[
m'_{ijk}= \sum_{p,q,r=0}^1 A_{ip}B_{jq}C_{kr}m_{pqr}.
\] 
Indeed, given any triple of 2-by-2 integer matrices $(A,B,C)$ (not necessarily in $\text{SL}_2(\mathbb{Z})^3$), this formula defines another Bhargava cube $M'$ from $M$. 

The \emph{Cayley hyperdeterminant} (or just \emph{hyperdeterminant}) of a Bhargava cube $M=(M_1 \mid M_2)$ is
\begin{align*}
\Det M &= \tfrac14\mleft(\det (M_1+M_2)-\det(M_1-M_2)\mright)^2-4\det M_1\det M_2\\
&= ((m_{000}m_{111}-m_{011}m_{100})+(m_{001}m_{110}-m_{010}m_{101}))^2\\
&\hspace*{5mm} -4(m_{000}m_{110}-m_{010}m_{100})(m_{001}m_{111}-m_{011}m_{101}).
\end{align*}

We write $\Det$ for hyperdeterminants and $\det$ for normal determinants (of square matrices). The $\text{SL}_2(\mathbb{Z})^3$ action preserves hyperdeterminants \cite{GeKaZe1994}. More generally, if $A$, $B$, and $C$, are 2-by-2 integer matrices and $M'=(A,B,C)M$, then $\Det M'=(\det A\det B\det C)^2\Det M$.

We are now in a position to define integer hypertilings, which are tensors $\mathbf{M}\colon \mathcal{I}\times \mathcal{J}\times \mathcal{K}\longrightarrow \mathbb{Z}$ subject to certain conditions. A subblock of a tensor is defined in much the same way as a subblock of a matrix.

\begin{definition}
For a nonzero integer $N$, an \emph{$N$-hypertiling} is a function $\mathbf{M}\colon \mathcal{I}\times \mathcal{J}\times \mathcal{K}\longrightarrow \mathbb{Z}$, for some index sets $\mathcal{I}$, $\mathcal{J}$, and $\mathcal{K}$, with the property that every 2-by-2-by-2 subblock has hyperdeterminant $N$. An \emph{integer hypertiling} is an $N$-hypertiling, for some value of $N$.
\end{definition}

Next we will define tame integer hypertilings. For this, we could mimic the definition of tame integer tilings by considering hyperdeterminants of 3-by-3-by-3 subblocks; however, it suits our purposes better to use cross sections of hypertilings. Formally, a \emph{cross section} of an integer hypertiling $\mathbf{M}\colon \mathcal{I}\times \mathcal{J}\times \mathcal{K}\longrightarrow \mathbb{Z}$ is the restriction of $\mathbf{M}$ to any one of $\{i_0\}\times \mathcal{J}\times \mathcal{K}$ (for some $i_0\in\mathcal{I}$), $\mathcal{I}\times \{j_0\}\times \mathcal{K}$ (for some $j_0\in\mathcal{J}$), or $\mathcal{I}\times \mathcal{J}\times \{k_0\}$ (for some $k_0\in\mathcal{K}$). Part of the definition of a tame integer hypertiling (to follow) is that each cross section must be a tame integer tiling. This includes the possibility that the cross section may be a tame 0-tiling; these are discussed in Section~\ref{section zero}.

In addition to cross sections, we need one further concept to handle some awkward hypertilings. Consider any 2-by-2-by-3 or 2-by-3-by-2 or 3-by-2-by-2 subblock $M$ of $\mathbf{M}$; let us assume the first case -- the others are similar. There are four 1-by-1-by-3 subblocks of $M$. Choose any two of these subblocks and stack them alongside each other to form a 3-by-2 matrix
\[
\begin{pmatrix}a&b\\c& d\\e&f\end{pmatrix}.
\]
We say that the integer hypertiling $\mathbf{M}$ is \emph{synchronised} if $ad-bc=cf-de$ for any 3-by-2 matrix formed in this way. If the two 1-by-1-by-3 subblocks belong to the same tame cross section, then this identity is automatically satisfied; it is only when the two subblocks are from different cross sections that the synchronised property offers an additional restriction.

\begin{definition}
An integer hypertiling $\mathbf{M}$ is \emph{tame} if it is synchronised and if each cross section of $\mathbf{M}$ is a tame integer tiling. 
\end{definition}

It is a straightforward exercise (which we omit) to show that every positive integer hypertiling with all cross sections integer tilings is necessarily tame.

The synchronised condition removes some pathological integer hypertilings. For example, consider the  49-hypertiling
\[
\setlength{\arraycolsep}{7pt}
\mleft(\begin{matrix*}[r]0 & \llap{$-$}1 & 0\\ 8 & 0 & \llap{$-$}8\\ 0 & 1 & 0\end{matrix*}\ \middle\vert\ \begin{matrix*}[r]0 & \llap{$-$}1 & 0\\ 1 & 0 & \llap{$-$}1\\ 0 & 1 & 0\end{matrix*}\ \middle\vert\ \begin{matrix*}[r]0 & \llap{$-$}7 & 0\\ 14 & 0 & \llap{$-$}14\\ 0 & 7 & 0\end{matrix*}\mright).
\]
Each cross section is a tame integer tiling, however, by considering  the 3-by-2 matrix
\[
\begin{pmatrix*}[r]8&  \!\!-1\\ 1 &  \!\!-1\\ 14 &  \!\!-7\end{pmatrix*} 
\]
we can see that this integer hypertiling is not synchronised.

For an example of a tame integer hypertiling, see Figure~\ref{figure6}, which is a 5-by-5-by-5 tame 1-hypertiling. The five cross sections displayed from left to right are tame $N$-tilings for values $N=1,6,15,2,10$. 

\begin{figure}[ht]
\centering\scriptsize
\(
\setlength{\arraycolsep}{4pt}
\renewcommand{\arraystretch}{1.4}
\mleft(\begin{matrix}4 & 11 & 7 & 10 & 13\\5 & 14 & 9 & 13 & 17\\1 & 3 & 2 & 3 & 4\\2 & 7 & 5 & 8 & 11\\1 & 4 & 3 & 5 & 7\end{matrix}\ \middle\vert\ \begin{matrix}10 & 28 & 18 & 26 & 34\\13 & 37 & 24 & 35 & 46\\3 & 9 & 6 & 9 & 12\\8 & 26 & 18 & 28 & 38\\5 & 17 & 12 & 19 & 26\end{matrix}\ \middle\vert\ \begin{matrix}16 & 45 & 29 & 42 & 55\\21 & 60 & 39 & 57 & 75\\5 & 15 & 10 & 15 & 20\\14 & 45 & 31 & 48 & 65\\9 & 30 & 21 & 33 & 45\end{matrix}\ \middle\vert\ \begin{matrix}6 & 17 & 11 & 16 & 21\\8 & 23 & 15 & 22 & 29\\2 & 6 & 4 & 6 & 8\\6 & 19 & 13 & 20 & 27\\4 & 13 & 9 & 14 & 19\end{matrix}\ \middle\vert\ \begin{matrix}14 & 40 & 26 & 38 & 50\\19 & 55 & 36 & 53 & 70\\5 & 15 & 10 & 15 & 20\\16 & 50 & 34 & 52 & 70\\11 & 35 & 24 & 37 & 50\end{matrix}\mright)
\)
\caption{A tame 1-hypertiling}
\label{figure6}
\end{figure}

Our first main result on integer hypertilings is a complete classification of all tame 1-hypertilings using triple Hadamard products of paths in $\mathscr{F}$. We use the subgroup 
\[
\{(I,I,I),(I,-I,-I),(-I,I,-I),(-I,-I,I)\}
\]
 of $\text{SL}_2(\mathbb{Z})^3$, where $I$ is the identity matrix. This subgroup is isomorphic to the Klein four-group $\mathbb{Z}_2\times \mathbb{Z}_2$, and it acts on triples of paths in $\mathscr{F}$.

\begin{maintheorem}{C}\label{theoremC}
The map
\[
(\mathbb{Z}_2\times \mathbb{Z}_2)\backslash \mleft\{\parbox{3.4cm}{\centering\textnormal{triples of paths in $\mathscr{F}$}}\mright\}\quad \longrightarrow\quad \mleft\{\parbox{3.2cm}{\centering\textnormal{tame
1-hypertilings}}\mright\}
\]
determined by 
\[
m_{ijk} = a_ic_je_k+b_id_jf_k, \qquad \text{$(i,j,k)\in\mathcal{I}\times\mathcal{J}\times\mathcal{K}$,}
\]
for paths $a_i/b_i$, $c_j/d_j$, and $e_k/f_k$ in $\mathscr{F}$, is a one-to-one correspondence.
\end{maintheorem}

Theorem~\ref{theoremC} will be proved in Section~\ref{section44}. In that section we will also see that every Bhargava cube of hyperdeterminant 1 is $\textnormal{SL}_2(\mathbb{Z})^3$-equivalent to
\[
\mleft(\begin{matrix}1 & 0 \\ 0 & 0 \end{matrix}\ \middle\vert\  \begin{matrix}0 & 0 \\ 0 & 1 \end{matrix}\mright).
\]
This result can be achieved using the methods of Bhargava \cite{Bh2004}; instead, however, we describe explicit triples in $\textnormal{SL}_2(\mathbb{Z})^3$ that achieve the equivalence.

For an example of the correspondence in Theorem~\ref{theoremC}, consider the  paths  $\gamma_1=\langle \tfrac12,\tfrac13,\tfrac01,\tfrac{-1}4,\tfrac{-1}{3}\rangle$,   $\gamma_2=\langle \tfrac21,\tfrac53,\tfrac32,\tfrac43,\tfrac54\rangle$,  and  $\gamma_3=\langle \tfrac11,\tfrac23,\tfrac35,\tfrac12,\tfrac25\rangle$ illustrated in Figure~\ref{figure7} (with vertices $a_i/b_i$, $c_j/d_j$, and $e_k/f_k$, respectively). The tame 1-hypertiling corresponding to these three paths is that shown in Figure~\ref{figure6}.

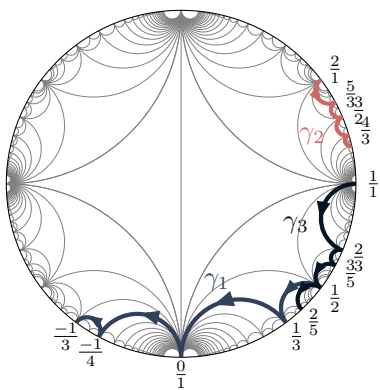
\begin{figure}[H]
\centering
\begin{tikzpicture}[scale=2.3]
\discfarey

\foreach \numA/\denA/\numB/\denB in {1/2/1/3,1/3/0/1,0/1/-1/4,-1/4/-1/3}{
	   \shline[gammacol,ultra thick](\numA:\denA:\numB:\denB);   
}
\foreach \num/\den in {1/2,1/3,0/1,-1/4,-1/3}{
	   \fareylabel(\num:\den); 
}

\foreach \numA/\denA/\numB/\denB in {2/1/5/3,5/3/3/2,3/2/4/3,4/3/5/4}{
	   \shline[deltacol,ultra thick](\numA:\denA:\numB:\denB);   
}
\foreach \num/\den in {2/1,5/3,3/2,4/3}{
	   \fareylabel(\num:\den); 
}

\foreach \numA/\denA/\numB/\denB in {1/1/2/3,2/3/3/5,3/5/1/2,1/2/2/5}{
	   \shline[alphacol,ultra thick](\numA:\denA:\numB:\denB);   
}
\foreach \num/\den in {1/1,2/3,3/5,2/5}{
	   \fareylabel(\num:\den); 
}

\node[gammacol] at (-70:0.6) {$\gamma_1$};
\node[deltacol] at (20:0.8) {$\gamma_2$};
\node[alphacol] at (-20:0.7) {$\gamma_3$};
   
\end{tikzpicture}
\caption{The paths $\gamma_1=\langle \tfrac12,\tfrac13,\tfrac01,\tfrac{-1}4,\tfrac{-1}{3}\rangle$,   $\gamma_2=\langle \tfrac21,\tfrac53,\tfrac32,\tfrac43,\tfrac54\rangle$,  and  $\gamma_3=\langle \tfrac11,\tfrac23,\tfrac35,\tfrac12,\tfrac25\rangle$ }
\label{figure7}
\end{figure}

Our fourth and final main result is about classifying all tame integer hypertilings. In this theorem we use the notation $u_{i0}/u_{i1}$, $v_{j0}/v_{j1}$, and $w_{k0}/w_{k1}$ (for $i\in \mathcal{I}, j\in \mathcal{J}, k\in \mathcal{K}$) to describe paths in $\mathscr{F}_R$, $\mathscr{F}_S$, and $\mathscr{F}_T$.  We say that a Bhargava cube $A$ is \emph{nonsingular} if $\Det A\neq 0$. 

\begin{maintheorem}{D}\label{theoremD}
Let $A$ be a nonsingular Bhargava cube. For any positive integers $R$, $S$, and $T$, there is a map
\[
\mleft\{\parbox{3.7cm}{\centering\textnormal{triples of minimal paths in $\mathscr{F}_R$, $\mathscr{F}_S$, and $\mathscr{F}_T$ }}\mright\}\quad \longrightarrow\quad \mleft\{\parbox{3.9cm}{\centering\textnormal{tame integer hypertilings}}\mright\}
\]
determined by 
\[
m_{ijk} = \sum_{p,q,r=0}^1A_{pqr}u_{ip}v_{jq}w_{kr},\qquad \text{$(i,j,k)\in\mathcal{I}\times\mathcal{J}\times\mathcal{K}$,}
\]
for minimal paths $u_{i0}/u_{i1}$, $v_{j0}/v_{j1}$,  $w_{k0}/w_{k1}$ in $\mathscr{F}_R$, $\mathscr{F}_S$,  $\mathscr{F}_T$. Furthermore, each tame integer hypertiling arises from some nonsingular cube $A$ and positive integers  $R$, $S$, $T$ in this way.
\end{maintheorem}

For a given nonsingular Bhargava cube $A$, and positive integers $R$, $S$, and $T$, the formula of Theorem~\ref{theoremD} determines a tame $N$-hypertiling, where $N=(RST)^2\Det A$. Every tame $N$-hypertiling arises in this way, although typically not all from the same cube $A$ (the case $N=1$ of Theorem~\ref{theoremC} is exceptional). Any Bhargava cube $A'$ that is  $\text{SL}_2(\mathbb{Z})^3$-equivalent to $A$ will determine exactly the same collection of tame $N$-hypertilings in Theorem~\ref{theoremD} as that of $A$.

Theorem~\ref{theoremD} requires $A$ to be nonsingular; there is a substantial literature on  three-dimensional lattices for which every 2-by-2-by-2 subblock is singular in the theory of discrete integrable systems. Much of this originated from Kashaev's study of certain quartic equations equivalent to vanishing hyperdeterminants in relation to the Ising model \cite{Ka1996}. Subsequent work includes that of Leaf \cite{Le2019}, who extends solutions of Kashaev's equations to the so-called $K$-hexahedron equations, involving additional variables at face centres.

There may be many triples of paths that determine the same hypertiling for a given Bhargava cube $A$. Precisely which triples depends on the stabilizer of $A$ in $\text{SL}_2(\mathbb{Z})^3$. We will see in Section~\ref{section44} that the stabilizer of a Bhargava cube with hyperdeterminant 1 is the group $\mathbb{Z}_2\times\mathbb{Z}_2$ discussed earlier; however, stabilizers in general need not be so simple. Theorem~\ref{theoremD} will be proved in Section~\ref{section66}.

When $R=S=T=1$, Theorem~\ref{theoremD} has special relevance to Bhargava's approach to binary quadratic forms. In this case the three paths belong to $\mathscr{F}$ and the formula for $m_{ijk}$ tells us that all Bhargava cubes in the integer hypertiling $\mathbf{M}$ are $\text{SL}_2(\mathbb{Z})^3$-equivalent. Bhargava \cite{Bh2004} extracts quadratic forms $Q_1$, $Q_2$, and $Q_3$ from a Bhargava cube $A$ by the formula $Q_i(x,y)=-\det(M_ix-N_iy)$, where $M_i$ and $N_i$ are parallel 2-by-2 slices of $A$. If these forms are primitive, then the discriminant of $Q_i$ is $\Det A$ and the sum $[Q_1]+[Q_2]+[Q_3]$ of $\text{SL}_2(\mathbb{Z})$-equivalence classes of the three forms is the identity in the class group of forms of discriminant $\Det A$. The local cubes of $\mathbf{M}$, all $\text{SL}_2(\mathbb{Z})^3$-equivalent, give a structured family of explicit representatives of $[Q_1]$, $[Q_2]$, and $[Q_3]$. For example, by considering the top corner cube
\[
\mleft(\begin{matrix}2 & 7 \\ 1 & 4 \end{matrix}\ \middle\vert\  \begin{matrix}8 & 26 \\ 5 & 17 \end{matrix}\mright)
\]
of Figure~\ref{figure6} we obtain representative quadratic forms $-x^2+5xy-6y^2$, $-2x^2+11xy-15y^2$, and $4x^2-7xy+3y^2$ for this tame integer hypertiling.

In the final section (Section~\ref{section55}) we discuss the observations of Demonet et al.\ \cite{DePlRuTu2018}. Those authors showed that there is essentially a unique bi-infinite positive integer hypertiling with the property that every cross section is an $\text{SL}_2$-tiling. (In fact, they established this in multiple dimensions.) Part of this hypertiling is shown in Figure~\ref{figure8}. The entries are Fibonacci numbers $F_n$ of odd index $n$, where $F_0=0$, $F_1=1$, and $F_{n+1}=F_n+F_{n-1}$, for $n\in\mathbb{Z}$. Each cross section is the same up to translations. Note that these cross sections are termed \emph{anti}-$\text{SL}_2$-tilings in  \cite{DePlRuTu2018} due to different conventions.

\begin{figure}[ht]
\centering
{\small
\[
\renewcommand*{\arraystretch}{1.4}
\hspace*{-0.7cm}\mleft.\begin{matrix}
34 & 13 & 5 & 2 & 1 & 1\\ 13&5 & 2 & 1 & 1 & 2\\ 5&2 & 1 & 1 & 2 & 5\\ 2&1& 1 & 2 & 5 & 13\\ 1&1 & 2 & 5 & 13 & 34 \\ 1&2 & 5 & 13 & 34 & 89  
\end{matrix}\quad
\middle\vert\quad
\begin{matrix}
13&5 & 2 & 1 & 1 & 2\\5 & 2 & 1 & 1 & 2 & 5\\ 2&1& 1 & 2 & 5 & 13\\ 1&1&2 & 5 & 13 & 34 \\ 1&2&5 & 13 & 34 & 89  \\ 2&5&13 & 34 & 89 & 233
\end{matrix}\quad
\middle\vert\quad
\begin{matrix}
5 & 2 & 1 & 1 & 2 & 5\\ 2&1& 1 & 2 & 5 & 13\\ 1&1&2 & 5 & 13 & 34 \\ 1&2&5 & 13 & 34 & 89  \\ 2&5&13 & 34 & 89 & 233\\ 5&13&34 & 89 & 233 & 610
 \end{matrix}\mright.
\]
}
\caption{Part of a positive integer hypertiling with $\text{SL}_2$ cross sections}
\label{figure8}
\end{figure}

This integer hypertiling is specified by the simple formula $m_{ijk}=F_{2(i+j+k)-1}$ (or, more generally, $m_{ijk}=F_{2(e+i+j+k)-1}$ for a translate of the hypertiling, for any integer $e$). From this formula one can see that the hypertiling is invariant under translations by triples $(i,j,k)$ where $i+j+k=0$ and that cross sections are positive integer $\text{SL}_2$-tilings, all translation equivalent. It is also obvious how the formula generalises to multiple dimensions.

This hypertiling can be described using three (identical) Farey paths $F_{2n-1}/F_{2n+1}$ and the Bhargava cube
\[
\mleft(\begin{matrix*}[r]3& -1\\ -1 & 0\end{matrix*}\ \middle\vert\ \begin{matrix*}[r]-1 & 0\\ 0 & 1\end{matrix*}\mright).
\]
We will see in Section~\ref{section55} that there are few other integer hypertilings with $\text{SL}_2$ cross sections, even when we allow negative integers.

\section{Paths and itineraries}

We recall that, given an index set $\mathcal{I}$, the set $\mathcal{I}^*$ is obtained from $\mathcal{I}$ by removing the greatest element from $\mathcal{I}$ (if it has one) and removing the least element from $\mathcal{I}$ (if it has one). 

\begin{definition}
The \emph{itinerary} of a path $(a_i/b_i)_{i\in \mathcal{I}}$ in $\mathscr{F}_R$ is the rational sequence 
\[
\lambda_i =\frac{1}{R}(a_{i-1}b_{i+1}-b_{i-1}a_{i+1}),
\]
for $i\in \mathcal{I}^*$.
\end{definition}

This definition of itinerary coincides with that of \cite{ShVaZa2025} in the special case $R=1$. The next lemma is similar to \cite{ShVaZa2025}*{Lemma~3.2}, so we abridge the proof.

\begin{lemma}\label{lemmak}
The itinerary of a path $\gamma$ in $\mathscr{F}_R$ with vertices $(a_i/b_i)_{i\in \mathcal{I}}$ is the unique sequence of rationals $(\mu_i)_{i\in \mathcal{I}^*}$  that satisfies
\[
a_{i-1}+a_{i+1}=\mu_ia_i\quad\text{and}\quad b_{i-1}+b_{i+1}=\mu_i b_i,
\]
for $i\in \mathcal{I}^*$.
\end{lemma}
\begin{proof}
Let $(\lambda_i)_{i\in \mathcal{I}^*}$ be the itinerary of $\gamma$. Then
\[
a_{i-1}+a_{i+1}=\frac{1}{R}(a_{i-1}(a_{i}b_{i+1}-b_{i}a_{i+1})+a_{i+1}(a_{i-1}b_{i}-b_{i-1}a_{i}))=\lambda_ia_i,
\]
and similarly $b_{i-1}+b_{i+1}=\lambda_ib_i$. The uniqueness assertion of the lemma is straightforward.
\end{proof}

In the classical Farey graph $\mathscr{F}_1$, a path is determined by its itinerary, up to $\text{SL}_2(\mathbb{Z})$ equivalence. This is \emph{not} true in $\mathscr{F}_R$, for $R>1$. For example, the paths $\langle \tfrac10,\tfrac02,\tfrac{-1}{2}\rangle$ and $\langle \tfrac20,\tfrac01,\tfrac{-2}{1}\rangle$ in $\mathscr{F}_2$ have the same itinerary -- namely the sequence with a single element 1 -- but they are not equivalent under $\text{SL}_2(\mathbb{Z})$ because $\tfrac10$ and $\tfrac20$ lie in distinct orbits. The next theorem states that \emph{minimal} paths in $\mathscr{F}_R$ with the same itinerary \emph{are} equivalent under $\text{SL}_2(\mathbb{Z})$, which explains the importance of the concept of minimal paths.

\begin{theorem}\label{theoremHD}
Let $\gamma$ and $\delta$ be two paths in $\mathscr{F}_R$ with the same itinerary. Suppose that $\gamma$ is minimal. Then there is a unique matrix $A\in\textnormal{SL}_2(\mathbb{Z})$ with $\delta=A\gamma$. In particular, $\delta$ is minimal also.
\end{theorem}
\begin{proof}
Let $\gamma$ have vertices $a_i/b_i$ and $\delta$ have vertices $c_i/d_i$, for $i\in \mathcal{I}$. We define
\[
A =\begin{pmatrix}c_0&c_1\\ d_0 & d_1\end{pmatrix}\begin{pmatrix}a_0&a_1\\ b_0 & b_1\end{pmatrix}^{-1}.
\]
Then $A\in\text{SL}_2(\mathbb{Q})$. Let $(\lambda_i)_{i\in\mathcal{I}^*}$ be the itinerary of $\gamma$ and $\delta$. By Lemma~\ref{lemmak}, each of the sequences $a_i$, $b_i$, $c_i$, and $d_i$ obey the same linear recurrence relations involving $\lambda_i$. Consequently, we have 
\[
\begin{pmatrix}c_i\\ d_i\end{pmatrix}=A\begin{pmatrix}a_i\\ b_i\end{pmatrix},\quad\text{for $i\in\mathcal{I}$}.
\]
Since $\gamma$ is minimal, the lattice generated by the pairs $(a_i,b_i)$ is the whole of $\mathbb{Z}^2$. Hence $A\mathbb{Z}^2\subseteq \mathbb{Z}^2$, so in fact $A\in\text{SL}_2(\mathbb{Z})$, as required. The uniqueness of $A$ is immediate.
\end{proof}

\section{Tame integer tilings}

We begin this section by adapting a well-known result from $\text{SL}_2$-tilings to $N$-tilings before establishing properties of tameness parameters of $N$-tilings that were promised in the introduction. Throughout, $N$ is a nonzero integer. 

The following lemma is elementary; we omit the proof.
 
\begin{lemma}\label{lemL}
Consider any 3-by-3 integer matrix
\[
A=\begin{pmatrix}a&b&c\\ d&e&f\\ g&h&i\end{pmatrix}
\]
with determinant 0 for which
\[
\det\!\begin{pmatrix}a&b\\ d&e\end{pmatrix}=\det\!\begin{pmatrix}b&c\\ e&f\end{pmatrix}=\det\!\begin{pmatrix}d&e\\ g&h\end{pmatrix}=\det\!\begin{pmatrix}e&f\\ h&i\end{pmatrix}=N.
\]
Then $a+c=(\Delta/N)b$, $d+f  = (\Delta/N) e$, and $g+i=(\Delta/N)h$, where $\Delta = af-cd=di-fg$.
\end{lemma}

The next theorem is similar to known results for $\text{SL}_2$-tilings (rather than $N$-tilings), notably \cite{ShVaZa2025}*{Theorem~6.1}. We give a sketch proof.

\begin{theorem}\label{theorem19}
An $N$-tiling $\mathbf{M}$ is tame if and only if there are sequences $(r_i)_{i\in \mathcal{I}^*}$ and $(s_j)_{j\in \mathcal{J}^*}$ in $\frac{1}{N}\mathbb{Z}$ such that
\begin{align*}
m_{i-1,j}+m_{i+1,j} &=r_im_{i,j},\quad\text{for $i\in \mathcal{I}^*,j\in \mathcal{J}$,}\\
m_{i,j-1}+m_{i,j+1}&=s_jm_{i,j},\quad\text{for $i\in \mathcal{I},j\in \mathcal{J}^*$.}
\end{align*}
Furthermore, if $\mathbf{M}$ is tame, then $(r_i)_{i\in \mathcal{I}^*}$ and $(s_j)_{j\in \mathcal{J}^*}$ are the unique sequences of rationals satisfying these recurrence relations.
\end{theorem}
\begin{proof}
Suppose that $\mathbf{M}$ is tame. Then, by Lemma~\ref{lemL}, we have
\[
m_{i,j-1}+m_{i,j+1}  = \frac{\Delta_{i,j}}{N}m_{i,j},\quad \text{for $i\in \mathcal{I}^*,j\in \mathcal{J}^*$,}
\]
where $\Delta_{i,j} = m_{i-1,j-1}m_{i,j+1}-m_{i-1,j+1}m_{i,j-1}=m_{i,j-1}m_{i+1,j+1}-m_{i,j+1}m_{i+1,j-1}$. This last pair of equations shows that, for any $j\in \mathcal{J}^*$, we have $\Delta_{i,j}=\Delta_{i',j}$, for all
$i,i'\in \mathcal{I}^*$. Consequently, there is $s_j\in \frac{1}{N}\mathbb{Z}$ with
\[
m_{i,j-1}+m_{i,j+1}=s_jm_{i,j}, \quad\text{for $i\in \mathcal{I}$.}
\]
The boundary cases $i\in \mathcal{I}\setminus\mathcal{I}^*$ of this recurrence follow from Lemma~\ref{lemL}. A similar argument gives the existence of the sequence $(r_i)_{i\in \mathcal{I}^*}$.

For the converse, the recurrence relations imply that any three rows (or columns) of a 3-by-3 subblock $A$ of $\mathbf{M}$ are linearly dependent, so $\det A=0$. Hence $\mathbf{M}$ is tame. The uniqueness assertion of the theorem is immediate.
\end{proof}
	
The next lemma shows that the vertical and horizontal tameness parameters are well defined (they do not depend on the pair of consecutive rows or columns chosen).

\begin{lemma}\label{lemmaG}
Let $\mathbf{M}$ be a tame $N$-tiling. For any indices $j,j'\in \mathcal{J}$, the integer
\[
m_{i,j}m_{i+1,j'} - m_{i,j'}m_{i+1,j}
\]
is independent of $i$. Similarly, for $i,i'\in \mathcal{I}$,  $m_{i,j}m_{i',j+1} - m_{i',j}m_{i,j+1}$ is independent of $j$.
\end{lemma}

The minor $m_{i,j}m_{i+1,j'} - m_{i,j'}m_{i+1,j}$ is illustrated below.

\begin{center}
\begin{tikzpicture}
\fill[rounded corners=2pt,fill=gammacol,fill opacity=0.3,xshift=0pt] (-2.5,-0.5) rectangle (-1.2,0.5);
\fill[rounded corners=2pt,fill=gammacol,fill opacity=0.3,xshift=105pt] (-2.5,-0.5) rectangle (-1.2,0.5);
\node at (0,0) {
\(
\begin{matrix} 
m_{i,j} & \cdot & \cdot &\cdot &\cdot &\cdot & m_{i,j'}\\
m_{i+1,j} & \cdot & \cdot &\cdot &\cdot &\cdot & m_{i+1,j'}
\end{matrix}
\)
};
\end{tikzpicture}	
\end{center}

\begin{proof}
Let $\ell_i=m_{i,j}m_{i+1,j'} - m_{i,j'}m_{i+1,j}$. Then, for $i\in\mathcal{I}^*$, we have
\[
\ell_i = m_{i,j}(r_im_{i,j'}-m_{i-1,j'})-m_{i,j'}(r_im_{i,j}-m_{i-1,j})=\ell_{i-1}.
\]
Hence $\ell_i$ is independent of $i$, and similarly  $m_{i,j}m_{i',j+1} - m_{i',j}m_{i,j+1}$ is independent of $j$.
\end{proof}

For the remainder of this section we write 
\(
\Delta(i,i';j,j')=m_{i,j}m_{i',j'}-m_{i,j'}m_{i',j}.
\) 

\begin{lemma}\label{lemmaQ}
Let $\mathbf{M}$ be a tame $N$-tiling. Let $i<i'$ and $j<j'$. Then
\[
N\Delta(i,i';j,j')= \Delta(i,i+1;j,j')\Delta(i,i';j,j+1).
\]
\end{lemma}
The minors $\Delta(i,i';j,j')$, $\Delta(i,i+1;j,j')$, and $\Delta(i,i';j,j+1)$ are illustrated below.

\begin{center}
\begin{tikzpicture}
\fill[rounded corners=2pt,fill=deltacol,fill opacity=0.3,yshift=0] (-3.15,1) rectangle (-0.5,1.5);
\fill[rounded corners=2pt,fill=deltacol,fill opacity=0.3,yshift=-2.55cm] (-3.15,1) rectangle (-0.5,1.5);
\fill[rounded corners=2pt,fill=alphacol,fill opacity=0.3,xshift=0] (-3.15,0.6) rectangle (-2,1.5);
\fill[rounded corners=2pt,fill=alphacol,fill opacity=0.3,xshift=5.05cm] (-3.15,0.6) rectangle (-2,1.5);
\draw[rounded corners=2pt,xshift=0] (-3.15,1) rectangle (-2,1.5);
\draw[rounded corners=2pt,xshift=5.05cm] (-3.15,1) rectangle (-2,1.5);
\draw[rounded corners=2pt,yshift=-2.55cm] (-3.15,1) rectangle (-2,1.5);
\draw[rounded corners=2pt,xshift=5.05cm,yshift=-2.55cm] (-3.15,1) rectangle (-2,1.5);
\node at (0,0) {
\(
\begin{matrix} 
m_{i,j} & m_{i,j+1} & \cdot & \cdot &\cdot &\cdot &\cdot & m_{i,j'}\\
m_{i+1,j} & \cdot & \cdot & \cdot &\cdot &\cdot &\cdot & m_{i+1,j'}\\
\cdot & \cdot & \cdot & \cdot &\cdot &\cdot &\cdot & \cdot \\
\cdot & \cdot & \cdot & \cdot &\cdot &\cdot &\cdot & \cdot \\
\cdot & \cdot & \cdot & \cdot &\cdot &\cdot &\cdot & \cdot \\
\cdot & \cdot & \cdot & \cdot &\cdot &\cdot &\cdot & \cdot \\
m_{i',j} & m_{i',j+1} & \cdot & \cdot &\cdot &\cdot &\cdot & m_{i',j'}
\end{matrix}
\)
};
\end{tikzpicture}	
\end{center}

\begin{proof}
We prove the result by induction on $n=(i'-i)+(j'-j)$, for $n=2,3,\dotsc$. When $n=2$ we must have $i'=i+1$ and $j'=j+1$, in which case the result follows immediately.

Suppose that $N\Delta(i,i';j,t)= \Delta(i,i+1;j,t)\Delta(i,i';j,j+1)$ for some values $i<i'$ and $t=j+1,j+2,\dots,j'$. We will prove that $N\Delta(i,i';j,j'+1)= \Delta(i,i+1;j,j'+1)\Delta(i,i';j,j+1)$. A
similar argument (which we omit) shows that $N\Delta(i,i'+1;j,j')= \Delta(i,i+1;j,j')\Delta(i,i'+1;j,j+1)$. Together these provide the inductive step that establishes the result.

Observe that $m_{t,j'+1}=\lambda m_{t,j'}-m_{t,j'-1}$, for $t=i,i+1,\dots,i'$, where $\lambda=s_{j'}$ from Theorem~\ref{theorem19}.
Then 
\begin{align*}
\Delta(i,i';j,j'+1) &=m_{i,j}m_{i',j'+1}-m_{i,j'+1}m_{i',j}\\
&=m_{i,j}(\lambda m_{i',j'}-m_{i',j'-1})-(\lambda m_{i,j'}-m_{i,j'-1})m_{i',j}\\
&= \lambda \Delta(i,i';j,j')-\Delta(i,i';j,j'-1).  
\end{align*}
Hence
\begin{align*}
&\Delta(i,i+1;j,j'+1)\Delta(i,i';j,j+1)\\
&\qquad\hspace*{44pt}= (\lambda \Delta(i,i+1;j,j')-\Delta(i,i+1;j,j'-1))\Delta(i,i';j,j+1)\\
&\qquad\hspace*{44pt}= \lambda N\Delta(i,i';j,j')-N\Delta(i,i';j,j'-1)\\
&\qquad\hspace*{44pt}=N\Delta(i,i';j,j'+1),
\end{align*}
as required.
\end{proof}
	
We can now prove a result about tameness parameters promised in the introduction. To state this theorem, it is helpful to define the \emph{total tameness parameter} $T$ of a tame $N$-tiling $\mathbf{M}$ to be the
quotient of $|N|$ by the greatest common divisor of all the 2-by-2 minors of $\mathbf{M}$.

\begin{theorem}\label{theoremV}
Let $\mathbf{M}$ be a tame $N$-tiling with vertical tameness parameter $R$, horizontal tameness parameter $S$, and total tameness parameter $T$. Then $RS=T$.
\end{theorem}
\begin{proof}
Let $r$ be the greatest common divisor of the 2-by-2 minors from two consecutive columns of $\mathbf{M}$. By Lemma~\ref{lemmaG}, this definition does not depend on which pair of columns we choose. Let $s$ be the greatest common divisor of the 2-by-2 minors from two consecutive rows of $\mathbf{M}$. Let $t$ be the greatest common divisor of all 2-by-2 minors of $\mathbf{M}$. When $N$ is negative, replace each of $r$, $s$, and $t$ by their negatives (in accordance with Section~\ref{section1}); then the signs of $r$, $s$, $t$, and $N$ coincide. We will prove that $rs=Nt$.

Observe that $\Delta(i,i+1;j,j')$ is divisible by $s$ and $\Delta(i,i';j,j+1)$ is divisible by $r$, for $i<i'$ and $j<j'$. By Lemma~\ref{lemmaQ}, $N\Delta(i,i';j,j')$ is divisible by $rs$. Hence $Nt$ is
divisible by $rs$.

Conversely, applying Lemma~\ref{lemmaQ} again, we see that $\Delta(i,i+1;j,j')\Delta(i,i';j,j+1)$ is divisible by $Nt$, for all $i<i'$ and $j<j'$. Remember from Lemma~\ref{lemmaG} that $\Delta(i,i+1;j,j')$ is
independent of $i$. Thus, by keeping $j$ and $j'$ fixed and varying $i$ and $i'$, we see that $r\Delta(i,i+1;j,j')$ is divisible by $Nt$. Then by varying $j$ and $j'$ we see that $rs$ is divisible
by $Nt$, as required.
\end{proof}

\section{Proof of Theorem~\ref{theoremA}}\label{section83}

This section comprises a proof of Theorem~\ref{theoremA}. We choose nonzero integers $K$, $L$, $R$, $S$, and $N$ with $K,R,S>0$ and $N=K^2LRS$. 

Consider two minimal paths $\gamma=(a_i/b_i)_{i\in \mathcal{I}}$ and $\delta=(c_j/d_j)_{j\in \mathcal{J}}$ in $\mathscr{F}_R$ and $\mathscr{F}_S$. Let $\mathbf{M}$ be the array with entries
\[
m_{i,j} =K(a_id_j-Lb_ic_j)= K\begin{pmatrix}a_i&b_i\end{pmatrix}J_L\begin{pmatrix}c_j\\d_j\end{pmatrix},\quad\text{where }
J_L = \begin{pmatrix}0&1\\ -L& 0\end{pmatrix}.
\]
Let us check that $\mathbf{M}$ is an $N$-tiling. We have
\[
\begin{pmatrix}
m_{i,j} & m_{i,j+1}\\
m_{i+1,j} & m_{i+1,j+1}
\end{pmatrix}
=K\begin{pmatrix}a_i&b_i\\a_{i+1}&b_{i+1} \end{pmatrix}J_L\begin{pmatrix}c_j & c_{j+1}\\d_j & d_{j+1}\end{pmatrix}.
\]
The determinant of the matrix product on the right is $K^2RLS$, which equals $N$. Hence $\mathbf{M}$ is indeed an $N$-tiling. Let us now prove that it is tame. By Lemma~\ref{lemmak}, the itinerary $(\lambda_i)_{i\in\mathcal{I}^*}$ of $\gamma$ satisfies $a_{i-1}+a_{i+1}=\lambda_ia_i$ and $b_{i-1}+b_{i+1}=\lambda_ib_i$, for $i\in \mathcal{I}^*$. Consequently, any three consecutive rows or columns of $\mathbf{M}$ are linearly dependent, so the determinant of any 3-by-3 subblock of $\mathbf{M}$ is 0.

Next, we need to show that the map is constant on $\Gamma^0(L)$ orbits of pairs of minimal paths. To see this, recall that with 
\[
A=\begin{pmatrix}a&bL\\ c & d\end{pmatrix}\in \Gamma^0(L) \quad\text{and}\quad B=\begin{pmatrix}a&b\\ cL & d\end{pmatrix} ,
\]
the action of $\Gamma^0(L)$ on $\mathscr{F}_R\times \mathscr{F}_S$ sends a pair $(x,y)$ to a pair $(x',y')$, where, regarding $x$, $y$, $x'$, and $y'$ as column vectors, we have $x'=Ax$ and $y'=By$. Then
\[
x'^TJ_Ly'  = x^TA^TJ_LBy = x^TJ_Ly,
\]
because $A^TJ_LB=J_L$, as one can easily verify. It follows that if two pairs of minimal paths $(\gamma,\delta)$ and $(\gamma',\delta')$ in $\mathscr{F}_R\times \mathscr{F}_S$ are $\Gamma^0(L)$ equivalent, then they
give rise to the same tame $N$-tiling $m_{i,j} =K(a_id_j-Lb_ic_j)$.

We now need to check that  $\mathbf{M}$ has tameness parameters $(K,L,R,S)$. Observe that, for $i,i+1\in \mathcal{I}$ and $j<j'$ in $\mathcal{J}$, 
\[
\begin{pmatrix}
m_{i,j} & m_{i,j'}\\
m_{i+1,j} & m_{i+1,j'}
\end{pmatrix}=K\begin{pmatrix}a_i&b_i\\a_{i+1}&b_{i+1} \end{pmatrix}J_L\begin{pmatrix}c_j & c_{j'}\\d_j & d_{j'}\end{pmatrix}.
\]
Taking determinants of both sides gives
\[
m_{i,j}m_{i+1,j'}-m_{i,j'}m_{i+1,j}=K^2LR(c_jd_{j'}-d_jc_{j'}).
\]
Since $\delta$ is minimal, the greatest common divisor of all the integers $c_jd_{j'}-d_jc_{j'}$ is 1. Hence, for $N>0$, the greatest common divisor of all the integers $m_{i,j}m_{i+1,j'}-m_{i,j'}m_{i+1,j}$ is $K^2LR$. Therefore the horizontal tameness parameter of $\mathbf{M}$ is $N/(K^2LR)=S$ -- and this also holds if $N<0$. Similarly, the vertical tameness parameter of $\mathbf{M}$ is $R$.

Next, suppose that all integers $a_id_j-Lb_ic_j$, for $i\in \mathcal{I},j\in \mathcal{J}$, are divisible by some prime number $p$. It cannot be that $L$ is divisible by $p$, for if that were so then either all $a_i$ would be divisible by $p$ or all $d_j$ would be divisible by $p$ and both those possibilities contradict minimality of one of $\gamma$ or $\delta$. From the matrix equation
\[
\frac{1}{K}
\begin{pmatrix}
m_{i,j} & m_{i,j'}\\
m_{i',j} & m_{i',j'}
\end{pmatrix}=\begin{pmatrix}a_i&b_i\\a_{i'}&b_{i'} \end{pmatrix}J_L\begin{pmatrix}c_j & c_{j'}\\d_j & d_{j'}\end{pmatrix}
\]
we see, by taking determinants, that $L(a_ib_{i'}-b_ia_{i'})(c_jd_{j'}-d_jc_{j'})$ is divisible by $p^2$, for all $i,i'\in \mathcal{I}$ and $j,j'\in \mathcal{J}$. It follows that all the minors of one of $\gamma$ or $\delta$ are divisible by $p$, which is a contradiction. Hence there are no nontrivial common factors of $a_id_j-Lb_ic_j$, for $i\in \mathcal{I},j\in \mathcal{J}$. Therefore the greatest common divisor of $\mathbf{M}$ is $K$. We have now proven that the map $m_{i,j}=K(a_id_j-Lb_ic_j)$ is well defined.

Our next task is to prove that the map is injective. Suppose that pairs of minimal paths $(\gamma,\delta)$ and $(\gamma',\delta')$ in $\mathscr{F}_R\times \mathscr{F}_S$ have the same image $\mathbf{M}$. Let
$\gamma$ have vertices $a_i/b_i$ and $\delta$ have vertices $c_j/d_j$. As usual we define $J=J_1$. Let
\[
U=\begin{pmatrix}
c_{0} & d_{0}\\
c_1 & d_1
\end{pmatrix}.
\]
Then $U^TJU=SJ$ and $J_LJJ_L^T=LJ$. Hence
\begin{align*}
\begin{pmatrix}
a_{i+1} & b_{i+1}
\end{pmatrix}
J
\begin{pmatrix}
a_{i-1} \\ b_{i-1}
\end{pmatrix}
&=\frac{1}{LS}\begin{pmatrix}
a_{i+1} & b_{i+1}
\end{pmatrix}
J_LU^TJUJ_L^T
\begin{pmatrix}
a_{i-1} \\ b_{i-1}
\end{pmatrix}\\
&=\frac{1}{K^2LS}\begin{pmatrix}
m_{i+1,0} & m_{i+1,1}
\end{pmatrix}
J
\begin{pmatrix}
m_{i-1,0} \\ m_{i-1,1}
\end{pmatrix}.
\end{align*}
Now, the itinerary of $\gamma$ is the sequence $(a_{i-1}b_{i+1}-b_{i-1}a_{i+1})/R$, for $i\in\mathcal{I}^*$, and because $(\gamma,\delta)$ and $(\gamma',\delta')$ have the same image $\mathbf{M}$, we see from the calculation above that $\gamma$ and $\gamma'$ have the same itineraries, as do $\delta$ and $\delta'$.

Next, by Theorem~\ref{theoremHD}, there exist matrices $A,B\in\text{SL}_2(\mathbb{Z})$ with $\gamma'=A\gamma$ and $\delta'=B\delta$. Looking at the two vertices  of the paths $\gamma$, $\gamma'$, $\delta$, and $\delta'$ indexed by 0 and 1 only, we see that
\[
\begin{pmatrix}
m_{0,0} & m_{0,1}\\
m_{1,0} & m_{1,1}
\end{pmatrix}
=K\begin{pmatrix}a_0&b_0\\ a_1 & b_1\end{pmatrix}J_L \begin{pmatrix}c_0&c_1\\ d_0 & d_1\end{pmatrix}
=K\begin{pmatrix}a_0&b_0\\ a_1 & b_1\end{pmatrix}A^T J_L B\begin{pmatrix}c_0&c_1\\ d_0 & d_1\end{pmatrix}.
\]
Hence $A^TJ_LB=J_L$. Let 
\[
A=\begin{pmatrix}a&b\\ c & d\end{pmatrix}.
\]
Then $B=J_L^{-1}(A^T)^{-1}J_L$; that is, 
\[
B = \frac{1}{L}\begin{pmatrix}0&-1\\ L& 0\end{pmatrix}\begin{pmatrix}d&-c\\ -b & a\end{pmatrix}\begin{pmatrix}0&1\\ -L& 0\end{pmatrix}=\begin{pmatrix}a&b/L\\ cL& d\end{pmatrix}.
\]
It follows that $b$ is divisible by $L$ -- let us say $b=b'L$, for $b'\in\mathbb{Z}$ -- so 
\[
A=\begin{pmatrix}a&b'L\\ c & d\end{pmatrix}\quad\text{and}\quad B=\begin{pmatrix}a&b'\\ cL& d\end{pmatrix}.
\]
Hence $(\gamma',\delta')$ and $(\gamma,\delta)$ lie in the same $\Gamma^0(L)$ orbit, as required.
	
It remains to prove that the map described in Theorem~\ref{theoremA} is surjective. To this end, let $\mathbf{M}$ be a tame $N$-tiling with tameness parameters $(K,L,R,S)$. Assume for the moment that $\mathbf{M}$
is of finite size, a $k$-by-$\ell$ matrix indexed by $\mathcal{I}=\{0,1,\dots,k-1\}$ and $\mathcal{J}=\{0,1,\dots,\ell-1\}$. Let $n$ be the lesser of $k$ and $\ell$. Using the Smith normal form of $\mathbf{M}$, we can find $A\in \text{GL}_k(\mathbb{Z})$ and $B\in \text{GL}_\ell(\mathbb{Z})$ with $\mathbf{M}=AXB$, where $X$ is the $k$-by-$\ell$ matrix with $X_{ii}=\lambda_i$, for $i=0,1,\dots,n-1$, and $X_{ij}=0$ otherwise. Here $\lambda_0=K$ and  $\lambda_1=K|L|$, and $\lambda_i=0$ for $i\geq 2$ because the rank of $\mathbf{M}$ is 2. Now, 
\[
\begin{pmatrix}0 & 1 \\ -L & 0\end{pmatrix}=\begin{pmatrix}1 & 0 \\ 0 & |L|\end{pmatrix}\begin{pmatrix}0 & 1 \\  -\sgn(L) & 0\end{pmatrix},
\]
so we can adjust $B$ (let us also call the adjusted matrix $B$) to give $\mathbf{M}=KAUB$, where $U$ is the $k$-by-$\ell$ matrix 
\[
U=\begin{pmatrix}\phantom{-}0&1&0&\cdots &0\\ -L &0& 0 &&\\ \phantom{-}0& 0& 0&&\\ \phantom{-}\vdots & & & \ddots& \\ \phantom{-}0 &&&&0\end{pmatrix}.
\]
Let $a_i=A_{i,0}$ and $b_i=A_{i,1}$ (the first two columns of $A$), and let $c_j=B_{0,j}$ and $d_j=B_{1,j}$ (the first two rows of $B$). Then from $\mathbf{M}=KAUB$ it follows that
\[
m_{i,j}=K(a_id_j-Lb_ic_j),\quad\text{for $i\in \mathcal{I}, j\in \mathcal{J}$}.
\]
We must show that $a_i/b_i$ and $c_j/d_j$ are minimal paths in $\mathscr{F}_R$ and $\mathscr{F}_S$. To see this, notice first that the greatest common divisor of the collection of all minors from the matrix
\[
\begin{pmatrix} a_{0} & a_{1} & \cdots & a_{k-1} \\ b_{0} & b_{1} & \cdots & b_{k-1}\end{pmatrix}
\]
must be 1, because these represent the first two columns of $A$, which lies in $\text{GL}_k(\mathbb{Z})$. Let $D_i=a_ib_{i+1}-b_ia_{i+1}$ and $E_j=c_jd_{j+1}-d_jc_{j+1}$. The determinant identity gives $N=K^2LD_iE_j$, so $D_iE_j=RS>0$. Consequently, $D_i$ is independent of $i$, $E_j$ is independent of $j$, and $D_i$ and $E_j$ have the same sign. For any $j,j+1\in \mathcal{J}$, we have
\[
\hspace*{-4em}\begin{pmatrix} m_{0,j} & m_{1,j} & \cdots & m_{k-1,j} \\ m_{0,j+1} & m_{1,j+1} & \cdots & m_{k-1,j+1}\end{pmatrix}^{\!T}
=K\begin{pmatrix} a_{0} & a_{1} & \cdots & a_{k-1} \\ b_{0} & b_{1} & \cdots & b_{k-1}\end{pmatrix}^{\!T}J_L\begin{pmatrix}c_j & c_{j+1} \\ d_j & d_{j+1}\end{pmatrix}.
\]
The greatest common divisor of the collection of all 2-by-2 minors from the matrix on the left is $|N|/R$. Hence, by taking 2-by-2 minors on the right, we see that 
\[
|E_j|=\frac{|N|}{K^2|L|R}=S.
\]
By replacing $b_i$ with $-b_i$ and $c_j$ with $-c_j$, for $i\in\mathcal{I},j\in \mathcal{J}$, we can assume that in fact $E_j=S$, in which case $D_i=R$. Therefore $a_i/b_i$ is a path in $\mathscr{F}_R$ and $c_j/d_j$ is a path in $\mathscr{F}_S$. Both are minimal, as we have seen. This concludes the proof in the finite case.

Suppose now that $\mathbf{M}\colon \mathcal{I}\times \mathcal{J}\longrightarrow\mathbb{Z}$ is an infinite tame $N$-tiling. Let $\mathcal{I}_1\subset \mathcal{I}_2\subset\dotsb$ and $\mathcal{J}_1\subset \mathcal{J}_2\subset\dotsb$ be sequences of finite intervals with
limits $\mathcal{I}$ and $\mathcal{J}$, respectively, and define $\mathbf{M}_n$ to be the finite tame $N$-tiling that is the restriction of $\mathbf{M}$ to $\mathcal{I}_n\times \mathcal{J}_n$. The $N$-tiling $\mathbf{M}_n$ has tameness parameters $(K_n,
L_n,R_n,S_n)$, say, so we can find minimal paths $a^{(n)}_i/b^{(n)}_i$ and $c^{(n)}_j/d^{(n)}_j$ in $\mathscr{F}_{R_n}$ and $\mathscr{F}_{S_n}$ with
\[
m_{i,j} = K_n(a^{(n)}_id^{(n)}_j-L_nb^{(n)}_ic^{(n)}_j),\quad\text{for $i\in \mathcal{I}_n,j\in \mathcal{J}_n$.}
\]
Recall that $R_n$ is the quotient of $|N|$ by the greatest common divisor of the 2-by-2 minors from two consecutive columns of $\mathbf{M}_n$. Hence the sequence $(R_n)$ is (not necessarily strictly) increasing, and so is $(S_n)$. Both $R_n$ and $S_n$ are divisors of $N$, so both sequences are eventually constant, with values $R$ and $S$, say. A similar argument shows that $(K_n)$ is a decreasing sequence, with limit $K$, say, and since $L_n=N/(K_n^2R_nS_n)$ we see that $(L_n)$ is eventually constant, with value $L$, say. Consequently, there exists an integer $n_0$ for which $(K_n, L_n,R_n,S_n)=(K,L,R,S)$ for $n\geq n_0$, and here $(K,L,R,S)$ are the tameness parameters of $\mathbf{M}$.

Now, for $n\geq n_0$, we have
\[
m_{i,j} = K(a^{(n)}_id^{(n)}_j-Lb^{(n)}_ic^{(n)}_j)= K(a^{(n+1)}_id^{(n+1)}_j-Lb^{(n+1)}_ic^{(n+1)}_j),
\]
for $i\in \mathcal{I}_n,j\in \mathcal{J}_n$. The itineraries of the paths $a^{(n)}_i/b^{(n)}_i$ and $a^{(n+1)}_i/b^{(n+1)}_i$ agree on $\mathcal{I}_n^*$ because they induce the same linear recurrence on $\mathbf{M}_n$, so Theorem~\ref{theoremHD} implies that the latter path is minimal. Similar comments apply to $c^{(n)}_j/d^{(n)}_j$ and $c^{(n+1)}_j/d^{(n+1)}_j$. We can now use the earlier injectivity argument to see that the two pairs of paths $(a^{(n)}_i/b^{(n)}_i,c^{(n)}_j/d^{(n)}_j)$ and $(a^{(n+1)}_i/b^{(n+1)}_i,c^{(n+1)}_j/d^{(n+1)}_j)$ are equivalent under $\Gamma^0(L)$. By applying an element of $\Gamma^0(L)$ we can ensure that in fact the latter sequence is equal to the former, for $i\in \mathcal{I}_n, j\in \mathcal{J}_n$. We then define $a_i/b_i=a^{(n)}_i/b^{(n)}_i$, for $n\geq n_0$, where $n$ is chosen to be sufficiently large that $i\in \mathcal{I}_n$. We define $c_j/d_j$ similarly. These two paths $a_i/b_i$ and $c_j/d_j$ are the required minimal paths in $\mathscr{F}_R$ and $\mathscr{F}_S$ to complete the proof.

\section{Positive integer tilings}\label{section14}

In this section we use Theorem~\ref{theoremA} to classify positive integer tilings in terms of pairs of paths. The approach is motivated by \cite{Sh2023}*{Theorem~1.2} for $\text{SL}_2$-tilings; indeed, that theorem is the special case $N=1$ of Theorem~\ref{theorem73}, to follow. For convenience, we assume that $N$ is positive, which requires only a little loss of generality -- after all, reflecting an $N$-tiling in any row or column gives an $(-N)$-tiling.

We endow the extended real line $\mathbb{R}_\infty=\mathbb{R}\cup\{\infty\}$ with the cyclic order inherited from the unit circle by stereographic projection. With this cyclic order, a decreasing sequence of real numbers is a clockwise sequence, as is, for example, the sequence $-1$, $\infty$, $1$. Let us say that a path $\gamma$ in $\mathscr{F}_R$ is a \emph{clockwise path} if the sequence of vertices forms a clockwise sequence in $\mathbb{R}_\infty$ that does not complete a full cycle of $\mathbb{R}_\infty$. Here we have adopted the sleight of hand that a vertex $a/b$ in $\mathscr{F}_R$, which was defined to be a pair $(a,b)$, is treated as the real number obtained from the quotient of $a$ by $b$. It is helpful to consider $a/b$ as both a formal pair and a real number, as one often does with fractions.

Any clockwise path $\gamma$ has forward and backward limit points $\gamma_\infty$ and $\gamma_{-\infty}$. If $\mathcal{I}$ is bounded above, then $\gamma_\infty$ is the final vertex of $\gamma$, and if $\mathcal{I}$ is unbounded above, then $\gamma_\infty$ is the forward limit in $\mathbb{R}_\infty$ of the sequence of vertices of $\gamma$.  We say that two clockwise paths $\gamma$ in $\mathscr{F}_R$ and $\delta$ in $\mathscr{F}_S$ are \emph{compatible} if they share no common vertices and $(\gamma_{-\infty},\gamma_{\infty},\delta_{-\infty},\delta_\infty)$ is a clockwise sequence in $\mathbb{R}_\infty$, where possibly $\gamma_{\infty}=\delta_{-\infty}$ and possibly  $\delta_{\infty}=\gamma_{-\infty}$. Now, given a positive integer $L$, and $\delta\in\mathscr{F}_S$ with vertices $c_j/d_j$, we define $L\delta$ to be the path in $\mathscr{F}_{LS}$ with vertices $(Lc_j)/d_j$. To apply Theorem~\ref{theoremA}, we are interested in the property that the pair $(\gamma,L\delta)$ in $\mathscr{F}_R\times\mathscr{F}_{LS}$ is compatible (\emph{not} the pair $(\gamma,\delta)$). We need to check that this property is invariant under the action of $\Gamma^0(L)$ on $\mathscr{F}_R\times \mathscr{F}_S$. This is a straightforward calculation which relies on the fact that elements of $\Gamma^0(L)$ act on $\mathbb{R}_\infty$ as orientation-preserving homeomorphisms. We omit the details. 

We say that an $N$-tiling is \emph{positive} if all its entries are positive and \emph{negative} if all its entries are negative. By Theorem~\ref{theoremA}, we can write the entries of a tame $N$-tiling $\mathbf{M}$ as $m_{i,j}=K(a_id_j-Lb_ic_j)$. Under this correspondence, exchanging the path $a_i/b_i$ for the path $(-a_i)/(-b_i)$ corresponds to replacing $\mathbf{M}$ with $-\mathbf{M}$. For this reason it is convenient to consider positive and negative $N$-tilings together when describing them using paths.

\begin{theorem}\label{theorem73}
For $N>0$, let $\mathbf{M}$ be a tame $N$-tiling with tameness parameters $(K,L,R,S)$ determined by the pair of minimal paths $\gamma$ in $\mathscr{F}_R$ and $\delta$ in $\mathscr{F}_S$. Then $\mathbf{M}$ is positive or  negative if and only if $\gamma$ and $L\delta$ are compatible clockwise paths.
\end{theorem}

Here the paths $\gamma$ and $\delta$ `determine' $\mathbf{M}$ according to the correspondence of Theorem~\ref{theoremA}. 

To illustrate Theorem~\ref{theorem73}, consider the minimal paths $\gamma=\langle \tfrac{10}{-3},\tfrac{7}{-2},\tfrac{4}{-1},\tfrac{1}{0},\tfrac{4}{1},\tfrac{7}{2},\tfrac{10}{3}\rangle\in\mathscr{F}_1$ and $\delta=\langle \tfrac{3}{4},\tfrac{6}{9},\tfrac{1}{2},\tfrac{0}{3},\tfrac{-1}{2},\tfrac{-6}{9},\tfrac{-3}{4}\rangle\in\mathscr{F}_3$ of Figure~\ref{figure4}. With $L=3$ we obtain
\[
3\delta=\langle \tfrac{9}{4},\tfrac{18}{9},\tfrac{3}{2},\tfrac{0}{3},\tfrac{-3}{2},\tfrac{-18}{9},\tfrac{-9}{4}\rangle\in\mathscr{F}_9.
\]
The paths $\gamma$ and $3\delta$, displayed in Figure~\ref{figure52}, are compatible clockwise paths, and, in agreement with Theorem~\ref{theorem73}, the corresponding tame $9$-tiling of Figure~\ref{figure1}(a) is positive.

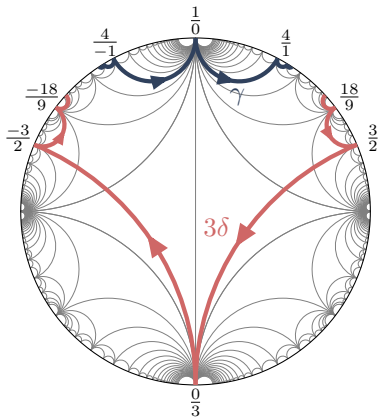
\begin{figure}[H]
\centering
\begin{tikzpicture}[scale=2.3]
\discfarey

\foreach \numA/\denA/\numB/\denB in {4/-1/1/0,4/-1/7/-2,7/-2/10/-3,1/0/4/1,4/1/7/2,7/2/10/3}{
	   \shline[gammacol,ultra thick](\numA:\denA:\numB:\denB);   
}

\foreach \numA/\denA/\numB/\denB in {9/4/18/9,18/9/3/2,3/2/0/3,0/3/-3/2,-3/2/-18/9,-18/9/-9/4}{
	   \shline[deltacol,ultra thick](\numA:\denA:\numB:\denB);   
}

\foreach \num/\den in {4/-1,1/0,4/1}{
	   \fareylabel(\num:\den); 
}
\foreach \num/\den in {18/9,3/2,0/3,-3/2,-18/9}{
	   \fareylabel(\num:\den); 
}

\node[deltacol] at (-37:0.15)  {$3\delta$};
\node[gammacol] at (70:0.7) {$\gamma$};
\end{tikzpicture}
\caption{Paths $\gamma=\langle \tfrac{10}{-3},\tfrac{7}{-2},\tfrac{4}{-1},\tfrac{1}{0},\tfrac{4}{1},\tfrac{7}{2},\tfrac{10}{3}\rangle\in\mathscr{F}_1$ and $3\delta=\langle \tfrac{9}{4},\tfrac{18}{9},\tfrac{3}{2},\tfrac{0}{3},\tfrac{-3}{2},\tfrac{-18}{9},\tfrac{-9}{4}\rangle\in\mathscr{F}_9$}
\label{figure52}
\end{figure}

\begin{proof}[Proof of Theorem~\ref{theorem73}]
Suppose first that $\gamma$ and $L\delta$ (with vertices $a_i/b_i$ and $Lc_j/d_j$) are compatible clockwise paths. The orbit of $\infty$ under $\Gamma^0(L)$ is dense in $\mathbb{R}_\infty$, so there is an element  of this orbit that lies clockwise of $Lc_0/d_0$ and anticlockwise of $Lc_1/d_1$. By mapping this element to $\infty$, and adjusting $\gamma$ and $L\delta$ accordingly, we can assume that $Lc_0/d_0<Lc_1/d_1$. Additionally, by applying the negative identity matrix $-I$ if need be (which belongs to $\Gamma^0(L)$), we can assume that $b_0>0$. We then have
\begin{equation}\label{equation61}
\frac{Lc_0}{d_0}<\frac{Lc_{-1}}{d_{-1}}<\dotsb <\frac{a_2}{b_2}<\frac{a_1}{b_1}<\frac{a_0}{b_0}<\frac{a_{-1}}{b_{-1}}<\frac{a_{-2}}{b_{-2}}<\dotsb<\frac{Lc_2}{d_2}<\frac{Lc_{1}}{d_{1}}.
\end{equation}
Observe that $b_{i-1}b_{i}$ is positive for all $i\in\mathcal{I}'$, because
\begin{equation}\label{equation62}
\frac{R}{b_{i-1}b_i}=\frac{a_{i-1}}{b_{i-1}}-\frac{a_{i}}{b_{i}}.
\end{equation}
Hence $b_i>0$, for $i\in\mathcal{I}$. By replacing $c_j/d_j$ with $(-c_j)/(-d_j)$ if need be (which changes $\delta$ and reverses the sign of every entry of $\mathbf{M}$) we can also assume that $d_0>0$. Similar arguments to those used already then show that $d_j>0$ for $j\leq 0$ and $d_j<0$ for $j>0$. Lastly, because $a_i/b_i >Lc_j/d_j$, for $i\in\mathcal{I}$ and $j\leq 0$, we see from the equation
\begin{equation}\label{equation63}
m_{i,j} = K(a_id_j-Lb_ic_j) = Kb_id_j\mleft(\frac{a_i}{b_i}-\frac{Lc_j}{d_j}\mright)
\end{equation}
that $m_{i,j}>0$. Similarly, $m_{i,j}>0$ if $j>0$. So $\mathbf{M}$ is positive. 

Conversely, suppose that $\mathbf{M}$ is positive or negative. By replacing $c_j/d_j$ with $(-c_j)/(-d_j)$ if need be we can assume that $\mathbf{M}$ is positive. Also, by applying a suitable element of $\Gamma^0(L)$, we can assume that $Lc_0/d_0<Lc_1/d_1$ and $d_0>0$. A short calculation then shows that $d_1<0$.

Consider \eqref{equation63} with $j=0,1$. Since $m_{i,0}>0$ and $m_{i,1}>0$, we have that $b_i>0$ and $Lc_0/d_0<a_i/b_i<Lc_1/d_1$, for $i\in\mathcal{I}$. We can then see from \eqref{equation62} that the sequence $a_i/b_i$ is decreasing.

Now, choose any $i\in\mathcal{I}$. From \eqref{equation63} we see that $Lc_j/d_j >a_i/b_i$ if and only if $d_j<0$. Let us prove by induction that $d_j<0$ for $j=1,2,\dotsc$. We know that $d_1<0$. Suppose that $d_{j-1}<0$, for some $j\geq 2$. If $d_j=0$, then $-d_{j-1}c_j=S$, so $c_j>0$, in which case $m_{i,j}=-KLb_ic_j<0$ -- a contradiction. If $d_j>0$, then 
\[
\frac{Lc_j}{d_j}=\frac{Lc_{j-1}}{d_{j-1}}-\frac{LS}{d_{j-1}d_j}>\frac{Lc_{j-1}}{d_{j-1}}>\frac{a_i}{b_i},
\]
which implies that $d_j<0$ -- another contradiction. So we must have $d_j<0$, as required. Hence $Lc_{j-1}/d_{j-1}-Lc_j/d_j=LS/(d_{j-1}d_j)>0$, for $j\geq 2$.

A similar argument shows that $d_j>0$, $a_i/b_i>Lc_j/d_j$, and $Lc_{j-1}/d_{j-1}>Lc_j/d_j$, for $j\leq 0$. Consequently, the chain of inequalities \eqref{equation61} is satisfied, so $\gamma$ and $L\delta$ are compatible clockwise paths.
\end{proof}

\section{Positive rational friezes}\label{section77}

In this section we prove Theorem~\ref{theoremB}. Our approach for positive rational friezes is in the same spirit as that of \cite{MoOvTa2015} for positive integer friezes. The definition of a frieze that we use is the standard one (termed a `regular frieze' in \cite{ShVaZa2025}).

\begin{definition}
A \emph{rational frieze} of width $n$ is a function 
\[
\mathbf{F}\colon \{(i,j)\in\mathbb{Z}\times\mathbb{Z}:0\leq i-j\leq n\}\longrightarrow \mathbb{Q}
\]
with entries $m_{i,j}=\mathbf{F}(i,j)$ such that  
\begin{itemize}
\item $m_{i,i}=m_{i+n,i}=0$, for $i\in\mathbb{Z}$ (top and bottom rows of zeros),
\item $m_{i+1,i}=m_{i+n,i+1}=1$, for $i\in\mathbb{Z}$ (second and second-last rows of ones),
\item $m_{i,j}m_{i+1,j+1}-m_{i,j+1}m_{i+1,j}=1$, for $0<i-j<n$ (diamond rule).
\end{itemize}
A \emph{positive rational frieze} is a rational frieze for which all nonzero entries are positive.
\end{definition}

We encountered an example of a positive rational frieze already, in Figure~\ref{figure5}, which has width 7. Coxeter proved that positive rational friezes (or indeed positive real friezes) are periodic, so the entries $m_{i,j}$ lie in $\tfrac{1}{N}\mathbb{Z}$, for some positive integer $N$. In this case we describe the positive rational frieze as a positive frieze over $\tfrac{1}{N}\mathbb{Z}$. 

It is well known and straightforward to prove that there is a unique way to extend a positive frieze $\mathbf{F}$ to a tame bi-infinite $\text{SL}_2$-tiling $\mathbf{M}$ (indexed by $\mathbb{Z}\times \mathbb{Z}$). The $\text{SL}_2$-tiling $\mathbf{M}$ has entries $m_{i,j}$  that satisfy $m_{i+n,j}=m_{i,j+n}=-m_{i,j}$, for $i,j\in\mathbb{Z}$. This allows us the liberty of regarding friezes as special types of $\text{SL}_2$-tilings.

Given any positive frieze over $\tfrac{1}{N}\mathbb{Z}$ of width $n$ -- thought of as an $\text{SL}_2$-tiling over $\tfrac{1}{N}\mathbb{Z}$ -- we can multiply each entry of the $\text{SL}_2$-tiling by $N$ to give an $N^2$-tiling with \emph{integral} entries. We define the \emph{greatest common divisor} of a positive frieze over $\tfrac{1}{N}\mathbb{Z}$ to be the greatest common divisor of this $N^2$-tiling. Notice that a positive frieze over $\tfrac{1}{N}\mathbb{Z}$ is also a positive frieze over $\tfrac{1}{M}\mathbb{Z}$ for any multiple $M$ of $N$ -- and the greatest common divisor depends on the choice of denominator. Typically it makes sense to work with the \emph{smallest} choice of $N$; however, we do not require that choice in our calculations.

\begin{lemma}\label{lemma81}
The tameness parameters $(K,L,R,S)$ of this $N^2$-tiling satisfy $L=1$ and $R=S$. 
\end{lemma}
\begin{proof}
Let us denote the entries of this bi-infinite integral $N^2$-tiling $\mathbf{M}$ by $m_{i,j}$, for $i,j\in\mathbb{Z}$. Coxeter \cite{Co1971} observed that positive friezes are not only periodic but in fact they have a glide reflection symmetry. This symmetry, inherited by the $N^2$-tiling, says that $m_{i,j}=m_{j,i-n}$, for $i,j\in\mathbb{Z}$. 

Let $t$ be the greatest common divisor of all the 2-by-2 minors of $\mathbf{M}$, and let $K$ be the greatest common divisor of all entries of $\mathbf{M}$. Clearly, $K^2\mid t$. Conversely, we observe that $t$ is a factor of the minor
\[
m_{i,j}m_{j,i+n}-m_{i,i+n}m_{j,j}=m_{i,j}^2,
\]
for any $i,j\in\mathbb{Z}$. Hence $t\mid K^2$. It follows that $t=K^2$. Recall that $N^2=K^2LRS$ and $RS=N^2/t$, by Theorem~\ref{theoremV}. Therefore $L=1$.

Equality of the vertical and horizontal tameness parameters follows immediately from the glide reflection property.
\end{proof}

We define a \emph{closed clockwise path} of length $n$ in $\mathscr{F}_R$ to be a path $\langle a_0/b_0,a_1/b_1,\dots,a_n/b_n\rangle$ in $\mathscr{F}_R$, where $a_n=-a_0$ and $b_n=-b_0$, and where the values of $a_0/b_0,a_1/b_1,\dots,a_{n-1}/b_{n-1}$ on $\mathbb{R}_\infty$ are in strict clockwise order and complete less than one cycle of $\mathbb{R}_\infty$. An example of a closed clockwise path is shown in Figure~\ref{figure46}(a). Any closed clockwise path can be extended to a bi-infinite path by defining $a_{i+n}=-a_i$ and $b_{i+n}=-b_i$, for $i\in\mathbb{Z}$. We will move freely between a finite closed clockwise path and its bi-infinite counterpart.

Let us now prove Theorem~\ref{theoremB}. We remark that the formula $m_{i,j}=\frac{1}{R}(a_jb_i-b_ja_i)$ in that theorem specifies a positive frieze over $\tfrac{1}{N}\mathbb{Z}$ of width $n$ even when $a_i/b_i$ is a closed clockwise path that is \emph{not} minimal. The requirement that $a_i/b_i$ is minimal is present in Theorem~\ref{theoremB} to give a clean one-to-one correspondence. For an indication of why this is necessary, observe that two non-minimal paths of the form $(2u_i)/v_i$ and $u_i/(2v_i)$, which typically are not $\text{SL}_2(\mathbb{Z})$-equivalent, have the same image under the map $m_{i,j}=\frac{1}{R}(a_jb_i-b_ja_i)$.

\begin{proof}[Proof of Theorem~\ref{theoremB}]
Observe that $1/R=K/N$. The map from paths to tilings determined by the formula $m_{i,j}=\tfrac{K}{N}(a_jb_i-b_ja_i)$ is, up to scaling by a factor $1/N$, the restriction of the one-to-one correspondence of Theorem~\ref{theoremA} (with $L=1$ and $S=R$) to the collection of those pairs of paths $(\gamma,-\gamma)$, where $\gamma$ is a minimal closed clockwise path in $\mathscr{F}_R$ of length $n$. It follows, then, that all we need to check is that the image of a pair $(\gamma,-\gamma)$ is a positive frieze over $\tfrac{1}{N}\mathbb{Z}$ of width $n$ and greatest common divisor $K$, and that the preimage of any positive frieze over $\tfrac{1}{N}\mathbb{Z}$ of width $n$ and greatest common divisor $K$ is such a pair $(\gamma,-\gamma)$.

For the first of these two tasks, let $\gamma$ be a minimal closed clockwise path in $\mathscr{F}_R$ of length $n$ with vertices $a_i/b_i$, for $i\in\mathbb{Z}$ (where $a_{i+n}=-a_i$ and $b_{i+n}=-b_i$). Then, by Theorem~\ref{theoremA} with $\delta=-\gamma$, the bi-infinite array $\mathbf{M}$ with entries $m_{i,j}=\tfrac{K}{N}(a_jb_i-b_ja_i)$ is a tame $\text{SL}_2$-tiling over $\tfrac{1}{N}\mathbb{Z}$. Let us check that $\mathbf{M}$ is in fact a rational frieze of width $n$. To do this, notice, for example, that 
\[
m_{i+n,i+1} = \frac{1}{R}(a_{i+1}b_{i+n}-b_{i+1}a_{i+n})=\frac{1}{R}(-a_{i+1}b_{i}+b_{i+1}a_{i})=1.
\]
The other properties of a rational frieze can be proved in a similar way. 

The entries of the frieze belong to $\tfrac{1}{N}\mathbb{Z}$. Since $Nm_{i,j}=K(a_jb_i-b_ja_i)$, and $\gamma$ is minimal, we see that the greatest common divisor of the integers $Nm_{i,j}$ is $K$. 

We must show that $m_{i,j}>0$, for $0<i-j<n$. For this, notice that we can assume after applying an element of $\text{SL}_2(\mathbb{Z})$ that $a_0/b_0>a_1/b_1>\dotsb >a_{n-1}/b_{n-1}$ and $b_0>0$. Then $b_i>0$, for $i=1,2,\dots,n-1$. Observe that 
\begin{equation}\label{equation47}
m_{i,j}=\frac{1}{R}(a_jb_i-b_ja_i)=\frac{b_ib_j}{R}\mleft(\frac{a_j}{b_j}-\frac{a_i}{b_i}\mright).
\end{equation}
Hence $m_{i,j}>0$, for $i,j\in\{0,1,\dots,n-1\}$ with $i>j$. The glide reflection symmetry ensures that $m_{i,j}>0$ for all $0<i-j<n$. Hence $\mathbf{M}$ is indeed a positive frieze over $\tfrac{1}{N}\mathbb{Z}$ of width $n$ and greatest common divisor $K$.

For the second of the two tasks, let $\mathbf{M}$ be a positive frieze over $\tfrac{1}{N}\mathbb{Z}$ of width $n$ and greatest common divisor $K$. After multiplying each entry of this frieze by $N$ we obtain a tame $N^2$-tiling $\mathbf{M}'$ with tameness parameters $(K,1,R,R)$, by Lemma~\ref{lemma81}. Let us write $m_{i,j}'$ for the $(i,j)$th entry of this $N^2$-tiling. From Theorem~\ref{theoremA} we can find minimal paths $a_i/b_i$ and $c_j/d_j$ in $\mathscr{F}_R$ with $m'_{i,j}=K(a_id_j-b_ic_j)$, for $i,j\in\mathbb{Z}$. Furthermore, after applying an element of $\text{SL}_2(\mathbb{Z})$, we can assume that $b_0=0$ and $a_0>0$.

Since $m'_{i,i}=0$ we see that $a_id_i=b_ic_i$,  so there exists $\lambda_i\in\mathbb{Q}$ with $c_i=\lambda_i a_i$ and $d_i=\lambda_i b_i$, for $i\in\mathbb{Z}$. We have $m'_{i+1,i}=N$ and $N=KR$, so 
\[
N = K(a_{i+1}d_i-b_{i+1}c_{i})=K\lambda_i(a_{i+1}b_i-b_{i+1}a_{i})=-KR\lambda_i.
\]
Therefore $\lambda_i=-1$, so $\delta=-\gamma$. From the equations $m'_{i+n,0}=-m'_{i,0}$ and $m'_{i+n,1}=-m'_{i,1}$ we see that $a_{i+n}=-a_i$ and $b_{i+n}=-b_i$. Also, since $m'_{i,0}=Ka_0b_i$, we have $b_i>0$, for $i=1,2,\dots,n-1$. Finally, from \eqref{equation47}, we obtain $a_1/b_1>a_2/b_2>\dotsb >a_{n-1}/b_{n-1}$, so $\gamma$ is a closed clockwise path, as required. 
\end{proof}

\section{Zero tilings}\label{section zero}

It is necessary for us to consider 0-tilings because they arise as cross sections of hypertilings. We choose the simplest definition of a tame 0-tiling to suit our purposes, which corresponds to setting $L=0$ in Theorem~\ref{theoremA}.

\begin{definition}
A \emph{$0$-tiling} is a function $\mathbf{M}\colon \mathcal{I}\times \mathcal{J}\longrightarrow \mathbb{Z}$ with the property that every 2-by-2 subblock has determinant $0$. A 0-tiling is \emph{tame} if there are positive integers $K$, $R$, and $S$ and minimal paths $a_i/b_i$ in $\mathscr{F}_R$ and $c_j/d_j$ in $\mathscr{F}_S$ with \(m_{i,j}=Ka_id_j\), for $i\in\mathcal{I},j\in\mathcal{J}$.
\end{definition}

Note that $R$ and $S$ are not necessarily unique.

To illustrate this definition, consider the 0-tilings
\[
\begin{pmatrix}
1 &0 &2\\
0 & 0 & 0\\
1 & 0 & 2
\end{pmatrix}
\qquad\text{and}\qquad
\begin{pmatrix*}[r]
2 &0 &\!\!\!-2\\
0 & 0 &0\\
\!-2 & 0 & 2
\end{pmatrix*}.
\]
The first matrix is the tensor product of the vectors $(1,0,1)$ and $(1,0,2)$; however, it is not tame because the sum of the top and bottom rows is not a scalar multiple of the middle row (see Lemma~\ref{lemma91}, to follow). The second matrix is a tame 0-tiling with tameness parameters $(2,0,1,1)$. It is given by the formula $m_{i,j}=2a_id_j$, where $a_i/b_i$ and $c_j/d_j$ are the vertices of the paths $\langle \tfrac10,\tfrac01,\tfrac{-1}{0}\rangle$ and $\langle \tfrac01,\tfrac{-1}0,\tfrac{0}{-1}\rangle$.

The next lemma is a version of Theorem~\ref{theorem19} for 0-tilings. 

\begin{lemma}\label{lemma91}
Let $\mathbf{M}$ be a tame 0-tiling. Then there are sequences of rationals $(r_i)_{i\in \mathcal{I}^*}$ and $(s_j)_{j\in \mathcal{J}^*}$ such that
\begin{align*}
 m_{i-1,j}+m_{i+1,j}&=r_im_{i,j},\quad\text{for $i\in\mathcal{I}^*,j\in \mathcal{J}$,}\\
m_{i,j-1}+m_{i,j+1}&=s_jm_{i,j},\quad\text{for $i\in\mathcal{I},j\in \mathcal{J}^*$.}
\end{align*}
\end{lemma}
\begin{proof}
We can write $m_{i,j}=Ka_id_j$ for paths $a_i/b_i$ in $\mathscr{F}_R$ and $c_j/d_j$ in $\mathscr{F}_S$. By Lemma~\ref{lemmak} there are sequences of rationals $r_i$ and $s_j$ with $a_{i-1}+a_{i+1}=r_i a_i$ and $d_{j-1}+d_{j+1}=s_j d_j$, for $i\in \mathcal{I}^*$ and $j\in \mathcal{J}^*$. The required result follows immediately.
\end{proof}

\section{Tame integer hypertilings}\label{section66}

We now turn our attention to integer hypertilings -- and prove Theorem~\ref{theoremD} (we prove Theorem~\ref{theoremC} in the next section). The more straightforward part of Theorem~\ref{theoremD} is that three paths determine an integer hypertiling, so let us begin with that.

\begin{theorem}\label{theorem71}
Let $A$ be a nonsingular Bhargava cube, and let $u_{i0}/u_{i1}$, $v_{j0}/v_{j1}$, and $w_{k0}/w_{k1}$  be three minimal paths in $\mathscr{F}_R$, $\mathscr{F}_S$, and $\mathscr{F}_T$. Then the function
\[
m_{ijk} = \sum_{p,q,r=0}^1 A_{pqr}u_{ip}v_{jq}w_{kr},\qquad \text{$(i,j,k)\in\mathcal{I}\times\mathcal{J}\times\mathcal{K}$,}
\]
is a tame $N$-hypertiling, where $N=R^2S^2T^2\Det A$.
\end{theorem}
\begin{proof}
Let $\mathbf{M}\colon\mathcal{I}\times\mathcal{J}\times\mathcal{K}\longrightarrow\mathbb{Z}$ be the array with entries $m_{i,j,k}$. Observe that the 2-by-2-by-2 subblock
\[
\mleft(\begin{matrix}
m_{i,j,k} & m_{i,j+1,k} \\
m_{i+1,j,k} & m_{i+1,j+1,k}
\end{matrix}
\quad
\middle\vert
\quad
\begin{matrix}
m_{i,j,k+1} & m_{i,j+1,k+1} \\
m_{i+1,j,k+1} & m_{i+1,j+1,k+1}
\end{matrix}\mright)
\]
is the image of $A$ under the triple of 2-by-2 integer matrices
\[
\mleft(\begin{pmatrix}u_{i,0} & u_{i,1}\\ u_{i+1,0} & u_{i+1,1}\end{pmatrix},\begin{pmatrix}v_{j,0} & v_{j,1}\\ v_{j+1,0} & v_{j+1,1}\end{pmatrix} ,\begin{pmatrix}w_{k,0} & w_{k,1}\\ w_{k+1,0} &w_{k+1,1}\end{pmatrix}\mright).
\]
Therefore it has hyperdeterminant $N=R^2S^2T^2\Det A$, so $\mathbf{M}$ is an $N$-hypertiling.

Next we show that each cross section of $\mathbf{M}$ is a tame integer tiling. Consider the restriction of $\mathbf{M}$ to $\mathcal{I}\times \mathcal{J}\times \{k_0\}$, for some integer $k_0\in \mathcal{K}$. Let $m_{i,j}=m_{i,j,k_0}$, for $i\in \mathcal{I}$ and
$j\in \mathcal{J}$. Let
\[
B = \begin{pmatrix} A_{0,0,0}w_{k_0,0}+A_{0,0,1}w_{k_0,1} & A_{0,1,0}w_{k_0,0}+A_{0,1,1}w_{k_0,1}\\ A_{1,0,0}w_{k_0,0}+A_{1,0,1}w_{k_0,1} & A_{1,1,0}w_{k_0,0}+A_{1,1,1}w_{k_0,1} \end{pmatrix}.
\]
Observe that $B$ is not the zero matrix because $A$ is nonsingular. We have
\[
m_{i,j} = \begin{pmatrix}u_{i,0} & u_{i,1}\end{pmatrix}B\begin{pmatrix}v_{j,0} \\ v_{j,1}\end{pmatrix}.
\]
Let $K$ be the greatest common divisor of $B$. Using the Smith normal form we can find $X,Y\in\text{SL}_2(\mathbb{Z})$ with 
\[
B=KX^T\begin{pmatrix}0 & 1\\-L & 0\end{pmatrix}Y,
\]
for some integer $L$. We define
\[
\begin{pmatrix}u_{i,0}' \\ u_{i,1}'\end{pmatrix}=X\begin{pmatrix}u_{i,0} \\ u_{i,1}\end{pmatrix}\quad\text{and}\quad\begin{pmatrix}v_{j,0}' \\ v_{j,1}'\end{pmatrix}=Y\begin{pmatrix}v_{j,0} \\ v_{j,1}\end{pmatrix}.
\]
Then $u_{i,0}'/u_{i,1}'$ and $v_{j,0}'/v_{j,1}'$ are minimal paths in $\mathscr{F}_R$ and $\mathscr{F}_S$. Observe that 
\[
m_{i,j}=K\begin{pmatrix}u_{i,0} & u_{i,1}\end{pmatrix}X^T\begin{pmatrix}0 & 1\\-L & 0\end{pmatrix}Y\begin{pmatrix}v_{j,0} \\ v_{j,1}\end{pmatrix}
=K(u_{i,0}'v_{j,1}'-Lu_{i,1}'v_{j,0}').
\]
Hence the restriction of $\mathbf{M}$ to $\mathcal{I}\times \mathcal{J}\times \{k_0\}$ is a tame $(K^2LRS)$-tiling, by Theorem~\ref{theoremA} (or by definition if $L=0$).

It remains to prove that $\mathbf{M}$ is synchronised. To do this, we will consider the 3-by-2 matrix
\[
\begin{pmatrix}
m_{i-1,j-1,k} &   m_{i-1,j,k-1}\\ 
m_{i,j-1,k} &  m_{i,j,k-1}\\
m_{i+1,j-1,k} &  m_{i+1,j,k-1}
\end{pmatrix}
\]
and prove that $m_{i-1,j-1,k}m_{i,j,k-1}-m_{i-1,j,k-1}m_{i,j-1,k}=m_{i,j-1,k}m_{i+1,j,k-1}-m_{i,j,k-1}m_{i+1,j-1,k}$. Other cases can be handled similarly. Now, by Lemma~\ref{lemmak}, there exists a rational $\mu_i$ with $(u_{i-1,0},u_{i-1,1})+(u_{i+1,0},u_{i+1,1})=\mu_i(u_{i,0},u_{i,1})$. Hence $m_{i-1,q,r}+m_{i+1,q,r}=\mu_i m_{i,q,r}$, for $q\in\mathcal{J}, r\in \mathcal{K}$. It follows that 
\begin{align*}
(m_{i-1,j-1,k}+m_{i+1,j-1,k})m_{i,j,k-1} &=\mu_im_{i,j,k-1}m_{i,j-1,k}\\
&=(m_{i-1,j,k-1}+m_{i+1,j,k-1})m_{i,j-1,k}.
\end{align*}
Rearranging this gives the required identity.
\end{proof}

Our next task is to prove that any tame integer hypertiling can be realised by three paths in Farey graphs. Some preliminary steps are required to reach this goal.

\begin{lemma}\label{lemma88}
Consider a 3-by-3-by-3 subblock
\[
M=\mleft(\begin{matrix}a_1 & a_2 & a_3\\ a_4 & a_5 & a_6\\ a_7 & a_8 & a_9\end{matrix}\ \middle\vert\ \begin{matrix}b_1 & b_2 & b_3\\ b_4 & b_5 & b_6\\ b_7 & b_8 & b_9\end{matrix}\ \middle\vert\ \begin{matrix}c_1 & c_2 & c_3\\ c_4 & c_5 & c_6\\ c_7 & c_8 & c_9\end{matrix}\mright)
\]
of a tame integer hypertiling $\mathbf{M}$. There exists a rational number $\lambda$ with  $a_i+c_i=\lambda b_i$, for $i=1,2,\dots,9$.
\end{lemma}
\begin{proof}
There are six vertical cross sections through $M$ with index sets $\{1,2,3\}$, $\{4,5,6\}$, $\{7,8,9\}$, $\{1,4,7\}$, $\{2,5,8\}$, and $\{3,6,9\}$. Let us call these \emph{vertical index sets}. We have three observations.
\begin{enumerate}
\item[(a)] Let $\mathcal{I}$ be a vertical index set. By Theorem~\ref{theorem19} and Lemma~\ref{lemma91}, there is a rational number $\lambda_\mathcal{I}$ with $a_i+c_i=\lambda_\mathcal{I} b_i$, for $i\in \mathcal{I}$.
\item[(b)] The rational $\lambda_\mathcal{I}$ from observation (a) is unique unless $b_i=0$, for $i\in \mathcal{I}$. In the exceptional case when $b_i=0$, for $i\in \mathcal{I}$, the rational $\lambda_\mathcal{I}$ can be chosen arbitrarily.
\item[(c)] Each index $i\in\{1,2,\dots,9\}$ belongs to precisely two vertical index sets, $\mathcal{I}$ and $\mathcal{J}$. If $b_i\neq 0$, then $\lambda_\mathcal{I}=(a_i+c_i)/b_i=\lambda_\mathcal{J}$.
\end{enumerate}

Suppose that $b_5\neq 0$. If any one of the four neighbouring entries $b_2$, $b_4$, $b_6$, or $b_8$ is nonzero, then one can show that all the rationals $\lambda_\mathcal{I}$ can be chosen to be equal, so there exists a rational $\lambda$ as required. We sketch a proof of this elementary assertion. Suppose without loss of generality that $b_2\neq 0$. If $b_1=b_4=b_7=0$, $b_3=b_6=b_9=0$, or $b_7=b_8=b_9=0$, then the existence of $\lambda$ is easily established. Otherwise, one of $b_1$ or $b_4$ is nonzero, one of $b_3$ or $b_6$ is nonzero, and one of $b_7$, $b_8$, or $b_9$ is nonzero. We can then apply observations (a)--(c) to deduce the existence of $\lambda$.

Suppose now that $b_5=0$. If any of the four corner entries $b_1$, $b_3$, $b_7$, or $b_9$ is nonzero, then, with similar reasoning to the previous paragraph, we can establish the existence of the required rational $\lambda$.

In summary, the lemma has been established for all cases apart from those in which the middle cross section of entries $b_i$ has one of the forms
\[
\begin{pmatrix}\ast & 0 & \ast \\ 0 & \ast & 0 \\ \ast & 0 & \ast  \end{pmatrix}\quad\text{or}\quad\begin{pmatrix}0 & \ast & 0 \\ \ast & 0 & \ast\\ 0 & \ast & 0  \end{pmatrix}.
\]
We will consider the second case; the first case can be handled similarly. In this second case, if any of the entries labelled $\ast$ are 0, then it is straightforward to establish the existence of the rational $\lambda$, so let us assume they are all nonzero. From observations (a)--(c) we can see that $\lambda_{1,2,3}=\lambda_{2,5,8}=\lambda_{7,8,9}$ and $\lambda_{1,4,7}=\lambda_{4,5,6}=\lambda_{3,6,9}$. We denote the first of these values by $\lambda$ and the second by $\mu$; we must show that $\lambda=\mu$.

Using the tameness of the cross sections and the synchronised property, we can write
\[
M=\mleft(\begin{matrix}x & q & -x \\ p & y & -p\\ -x & -q & x \end{matrix}\quad \middle\vert\quad \begin{matrix}0 & b & 0 \\ a & 0 & -a \\ 0 & -b & 0 \end{matrix}\quad \middle\vert\quad \begin{matrix}-x & s & x \\ r & -y & -r\\ x & -s & -x \end{matrix}\mright),
\]
for some integers $a$, $b$, $x$, $y$, $p$, $q$, $r$, and $s$. Consider the 3-by-2 matrix
\[
A=\begin{pmatrix}p & q \\ a &  b \\ r &  s  \end{pmatrix}.
\]
By the synchronised property we have $pb-qa=as-br$, so 
\[
\mu= \frac{p+r}{a} = \frac{q+s}{b}=\lambda,
\]
as required.
\end{proof}

Lemma~\ref{lemma88} established the existence of a rational number $\lambda$ subject to certain properties; in fact, there is only one such rational number that satisfies these properties. The next lemma is a stronger uniqueness assertion, for smaller subblocks.

\begin{lemma}\label{lemma89}
Consider a 2-by-2-by-3 subblock
\[
M=\mleft(\begin{matrix}a_1 & a_2 \\ a_3 & a_4\end{matrix}\ \middle\vert\ \begin{matrix}b_1 & b_2 \\ b_3 & b_4\end{matrix}\ \middle\vert\ \begin{matrix}c_1 & c_2 \\ c_3 & c_4\end{matrix}\mright)
\]
of a tame integer hypertiling $\mathbf{M}$. There exists a unique rational number $\lambda$ with  $a_i+c_i=\lambda b_i$, for $i=1,2,3,4$.
\end{lemma}
\begin{proof}
The existence of $\lambda$ was established in Lemma~\ref{lemma88}; we have only to prove uniqueness. This is immediate, for we must have $b_i\neq 0$, for some $i\in\{1,2,3,4\}$, in which case $\lambda = (a_i+c_i)/b_i$.
\end{proof}

The following theorem is a version of Theorem~\ref{theorem19} for integer hypertilings rather than tilings. There is a converse to the theorem (with some modifications); we omit this because we do not need it.

\begin{theorem}\label{theorem69}
Let $\mathbf{M}$ be a tame integer hypertiling. Then there are sequences of rationals $(r_i)_{i\in \mathcal{I}^*}$, $(s_j)_{j\in \mathcal{J}^*}$, and $(t_k)_{k\in \mathcal{K}^*}$ such that
\begin{align*}
m_{i-1,j,k}+m_{i+1,j,k}=r_im_{i,j,k},\quad \text{for $i\in \mathcal{I}^*,j\in \mathcal{J}, k\in \mathcal{K}$,}\\
m_{i,j-1,k}+m_{i,j+1,k}=s_jm_{i,j,k},\quad \text{for $i\in \mathcal{I}, j\in \mathcal{J}^*,k\in \mathcal{K}$,}\\
m_{i,j,k-1}+m_{i,j,k+1}=t_km_{i,j,k},\quad \text{for $i\in \mathcal{I}, j\in \mathcal{J},k\in \mathcal{K}^*$.}
\end{align*}
Furthermore, the sequences $r_i$, $s_j$, and $t_k$ are uniquely determined by $\mathbf{M}$.
\end{theorem}
\begin{proof}
 Let us prove the existence and uniqueness of the sequence $t_k$; the other two cases are similar. Consider some integer $k\in \mathcal{K}^*$. Let $M_{i,j}=(A_{k-1}\mid A_k\mid A_{k+1})$, where
\[
A_\ell=\begin{pmatrix}m_{i-1,j-1,\ell} & m_{i-1,j,\ell} & m_{i-1,j+1,\ell}\\ m_{i,j-1,\ell} & m_{i,j,\ell} & m_{i,j+1,\ell}\\ m_{i+1,j-1,\ell} & m_{i+1,j,\ell} & m_{i+1,j+1,\ell}\end{pmatrix}.
\]
By Lemma~\ref{lemma88}, there exists $\lambda_{i,j}\in \mathbb{Q}$ with $m_{p,q,k-1}+m_{p,q,k+1}=\lambda_{i,j}m_{p,q,k}$, for $p\in\{i-1,i,i+1\}$ and $q\in\{j-1,j,j+1\}$. From considering the overlap of the subblocks $M_{i,j}$ and $M_{i,j+1}$ and applying Lemma~\ref{lemma89}, we deduce that $\lambda_{i,j}=\lambda_{i,j+1}$. Arguing in the same way we can see that there exists a rational $t_k$ such that $\lambda_{i,j}=t_k$, for all $i\in \mathcal{I}^*$, $j\in \mathcal{J}^*$, which is unique, by Lemma~\ref{lemma89}.
\end{proof}

We remarked in the introduction that there is an action of $\text{SL}_2(\mathbb{Z})^3$ on Bhargava cubes $\mathbb{Z}^2\otimes\mathbb{Z}^2\otimes\mathbb{Z}^2$. Here we need this same action in higher dimensions. Specifically, let us assume that $k$, $\ell$, and $m$ are positive integers, all at least 2. Then $\text{SL}_k(\mathbb{Z})\times\text{SL}_\ell(\mathbb{Z})\times\text{SL}_m(\mathbb{Z})$ acts on $\mathbb{Z}^k\otimes\mathbb{Z}^\ell\otimes\mathbb{Z}^m$ by the formula $\mathbf{M}'=(U,V,W)\mathbf{M}$, where $U\in\text{SL}_k(\mathbb{Z})$, $V\in\text{SL}_\ell(\mathbb{Z})$, $W\in\text{SL}_m(\mathbb{Z})$, and $\mathbf{M},\mathbf{M}'\in \mathbb{Z}^k\otimes\mathbb{Z}^\ell\otimes\mathbb{Z}^m$, and the entries of $\mathbf{M}'$ are
\[
m_{ijk}'= \sum_{(p,q,r)\in \mathcal{I}\times \mathcal{J}\times \mathcal{K}} U_{ip}V_{jq}W_{kr}m_{pqr}.
\]
For convenience, we assume in the next lemma that the index sets of $\mathbf{M}$ and $\mathbf{M}'$ are $\mathcal{I}=\{0,1,\dots,k-1\}$, $\mathcal{J}=\{0,1,\dots,\ell-1\}$, and $\mathcal{K}=\{0,1,\dots,m-1\}$.	

\begin{lemma}\label{lemma28}
Let $\mathbf{M}$ be a $k$-by-$\ell$-by-$m$ tame integer hypertiling. There exists $U\in\textnormal{SL}_k(\mathbb{Z})$, $V\in\textnormal{SL}_\ell(\mathbb{Z})$, and $W\in\textnormal{SL}_m(\mathbb{Z})$ such that the tensor $\mathbf{M}'=(U,V,W)\mathbf{M}$ satisfies $m_{ijk}'=0$ whenever $i$, $j$, or $k$ is greater than 1.
\end{lemma}
\begin{proof}
Let $v_i$ denote the $\ell$-by-$m$ integer matrix with entries $m_{iqr}$, for $q=0,1,\dots,\ell-1$ and $r=0,1,\dots,m-1$. By Theorem~\ref{theorem69}, there are integers $\eta_i$, $\lambda_i$, and $\mu_i$ (where $\eta_i\neq 0$) with $\eta_iv_i=\lambda_i v_0+\mu_i v_1$ and $\gcd(\eta_i,\lambda_i,\mu_i)=1$, for $i=2,3,\dots,k-1$. Let $A$ be the $(k-2)$-by-$k$ matrix
\[
A = \begin{pmatrix} -\lambda_2 & -\mu_2 & \eta_2 &&&\\ -\lambda_3 & -\mu_3 & & \eta_3 &&\\   \vdots & \vdots &&& \ddots &\\ -\lambda_{k-1}& - \mu_{k-1}&&&& \eta_{k-1}\end{pmatrix}.
\]
Then $\sum_{p=0}^{k-1}A_{ip}v_p=0$, for $i=0,1,\dots,k-3$. Using the Smith normal form we can find $X\in \text{SL}_{k-2}(\mathbb{Z})$ and $Y\in \text{SL}_{k}(\mathbb{Z})$ with 
\[
A =X \begin{pmatrix}0 & 0 & d_2&&&\\ 0&0&& d_3 &&\\ \vdots&\vdots&&& \ddots & \\  0 & 0&&&& d_{k-1} \end{pmatrix}Y,
\]
for nonzero integers $d_2,d_3,\dots,d_{k-1}$. Hence, for $i\geq 2$, we have
\[
\sum_{p=0}^{k-1}Y_{ip}v_p = \frac{1}{d_{i}}\sum_{p=0}^{k-1} (X^{-1}A)_{i-2,p}v_p = 0.
\]
We choose $U=Y$. Similarly we can find $V\in\textnormal{SL}_\ell(\mathbb{Z})$ and  $W\in\textnormal{SL}_m(\mathbb{Z})$ with $\sum_q V_{jq}m_{iqk}=0$ (for $j\geq 2$) and $\sum_r W_{kr}m_{ijr}=0$ (for $k\geq 2$) . Then $(U,V,W)$ is the required triple.
\end{proof}

We need one final lemma before we can prove Theorem~\ref{theoremD}.

\begin{lemma}\label{lemma47}
Let $A$ be a nonsingular Bhargava cube and let $M$ be a 3-by-2-by-2 subblock of a tame integer hypertiling with entries $m_{ijk}=\sum_{p=0}^1 u_{ip}A_{pjk}$, for $i\in\{0,1,2\}$ and $j,k\in\{0,1\}$. Then
$u_{00}u_{11}-u_{01}u_{10}=u_{10}u_{21}-u_{11}u_{20}$.
\end{lemma}
\begin{proof}
Since $M$ is a subblock of a tame integer hypertiling we have 
\[
m_{0jk}m_{1j'k'}-m_{1jk}m_{0j'k'}=m_{1jk}m_{2j'k'}-m_{2jk}m_{1j'k'},
\]
for $j,j',k,k'\in\{0,1\}$. Now, $m_{0jk}m_{1j'k'}-m_{1jk}m_{0j'k'}$ equals
\[
\hspace*{-15mm}(u_{00}A_{0jk}+u_{01}A_{1jk})(u_{10}A_{0j'k'}+u_{11}A_{1j'k'})-(u_{10}A_{0jk}+u_{11}A_{1jk})(u_{00}A_{0j'k'}+u_{01}A_{1j'k'}),
\]
and this simplifies to 
\[
(u_{00}u_{11}-u_{01}u_{10})(A_{0jk}A_{1j'k'}-A_{1jk}A_{0j'k'}).
\]
Similarly $m_{1jk}m_{2j'k'}-m_{2jk}m_{1j'k'}=(u_{10}u_{21}-u_{11}u_{20})(A_{0jk}A_{1j'k'}-A_{1jk}A_{0j'k'})$. Since $\Det A\neq 0$, we see from the formula for $\Det A$ stated in the introduction that there must be a quadruple $j,j',k,k'\in\{0,1\}$ for which $A_{0jk}A_{1j'k'}-A_{1jk}A_{0j'k'}\neq 0$. Consequently, $u_{00}u_{11}-u_{01}u_{10}=u_{10}u_{21}-u_{11}u_{20}$, as required.
\end{proof}

We can now show that any tame integer hypertiling can be described by three paths in Farey graphs, the reverse statement of Theorem~\ref{theorem71}.

\begin{theorem}\label{theorem72}
Let $\mathbf{M}$ be a tame N-hypertiling. Then there are positive integers $R$, $S$, and $T$, minimal paths $u_{i0}/u_{i1}$, $v_{j0}/v_{j1}$, and $w_{k0}/w_{k1}$ in $\mathscr{F}_R$, $\mathscr{F}_S$, and $\mathscr{F}_T$, and a Bhargava cube $A$ with $N=R^2S^2T^2\Det A$ such that
\[
m_{ijk} = \sum_{p,q,r=0}^1A_{pqr}u_{ip}v_{jq}w_{kr},
\]
for $i\in \mathcal{I}, j\in \mathcal{J}, k\in \mathcal{K}$.
\end{theorem}
\begin{proof}
Let us assume for the moment that $\mathbf{M}$ is of finite size $k$-by-$\ell$-by-$m$  indexed by $\mathcal{I}=\{0,1,\dots,k-1\}$, $\mathcal{J}=\{0,1,\dots,\ell-1\}$, and $\mathcal{K}=\{0,1,\dots,m-1\}$.

By Lemma~\ref{lemma28}, we can find $U\in\textnormal{SL}_k(\mathbb{Z})$, $V\in\textnormal{SL}_\ell(\mathbb{Z})$, and $W\in\textnormal{SL}_m(\mathbb{Z})$ such that $\mathbf{M}'=(U^{-1},V^{-1},W^{-1})\mathbf{M}$ satisfies $m_{ijk}'=0$ whenever $i$, $j$, or $k$ is greater than 1. Let $(u_{i0},u_{i1})=(U_{i0},U_{i1})$ (the first two columns of $U$), $(v_{j0},v_{j1})=(V_{j0},V_{j1})$ (the first two columns of $V$), and $(w_{k0},w_{k1})=(W_{k0},W_{k1})$ (the first two columns of $W$). Let $A$ be the Bhargava cube with $A_{pqr}=m_{pqr}'$, for $p,q,r\in\{0,1\}$. We have $\mathbf{M}=(U,V,W)\mathbf{M}'$. So 
\[
m_{ijk} = \sum_{p,q,r=0}^1A_{pqr}u_{ip}v_{jq}w_{kr}.
\]
We must show that $u_{i0}/u_{i1}$, $v_{j0}/v_{j1}$, and $w_{k0}/w_{k1}$ are minimal paths in $\mathscr{F}_R$, $\mathscr{F}_S$, and $\mathscr{F}_T$, for some positive integers $R$, $S$, and $T$. Let us do this for $u_{i0}/u_{i1}$. Let $B$ be the Bhargava cube with entries $B_{ijk}=\sum_{q,r=0}^1A_{iqr}v_{jq}w_{kr}$, for $i,j,k\in\{0,1\}$. Then $m_{ijk}=\sum_{p=0}^1B_{pjk}u_{ip}$, so $B$ is nonsingular, because $\mathbf{M}$ is an $N$-hypertiling. By Lemma~\ref{lemma47}, we have $u_{i-1,0}u_{i,1}-u_{i-1,1}u_{i,0}=u_{i,0}u_{i+1,1}-u_{i,1}u_{i+1,0}$, for all $i\in\mathcal{I}^*$. By replacing $u_{i0}$ by $-u_{i0}$ for all indices $i$, if need be (and adjusting $A$ accordingly), we can assume that in fact $u_{i-1,0}u_{i,1}-u_{i-1,1}u_{i,0}=R$, for some positive integer $R$ and all $i\in\mathcal{I}'$ (where $R\neq0$ because $\mathbf{M}$ is an $N$-hypertiling). Hence $u_{i0}/u_{i1}$ is a path in $\mathscr{F}_R$. It is a minimal path because it is formed from the first two columns of an element of $\text{SL}_k(\mathbb{Z})$.  

Similar arguments show that $v_{j0}/v_{j1}$ and $w_{k0}/w_{k1}$ are minimal paths in $\mathscr{F}_S$ and $\mathscr{F}_T$, for some positive integers $S$ and $T$. As a consequence, we have  $N=R^2S^2T^2\Det A$, and this concludes the proof in the finite case. 

The argument to go from the finite case to the infinite case is similar to that used in proving Theorem~\ref{theoremA}. Let $\mathbf{M}\colon \mathcal{I}\times \mathcal{J}\times\mathcal{K}\longrightarrow\mathbb{Z}$ be an infinite tame $N$-hypertiling. Let $\mathcal{I}_1\subset \mathcal{I}_2\subset\dotsb$, $\mathcal{J}_1\subset \mathcal{J}_2\subset\dotsb$, and $\mathcal{K}_1\subset \mathcal{K}_2\subset\dotsb$ be sequences of finite intervals with limits $\mathcal{I}$, $\mathcal{J}$, and $\mathcal{K}$, respectively, and let $\mathbf{M}_n$ be the restriction of $\mathbf{M}$ to $\mathcal{I}_n\times \mathcal{J}_n\times \mathcal{K}_n$. The recurrence relations and synchronisation identities for $\mathbf{M}$ restrict to $\mathbf{M}_n$, so the finite construction just supplied applies to $\mathbf{M}_n$. We define $R_n$, $S_n$, and $T_n$ to be the parameters of the graphs for the corresponding minimal paths of $\mathbf{M}_n$. By passing to subsequences we can assume that in fact all these sequences are constant (since they are all divisors of $N$), with values $R$, $S$, and $T$. 

It follows that we can find minimal paths $u_{i0}^{(n)}/u_{i1}^{(n)}$, $v_{j0}^{(n)}/v_{j1}^{(n)}$, and $w_{k0}^{(n)}/w_{k1}^{(n)}$ in $\mathscr{F}_R$, $\mathscr{F}_S$, and $\mathscr{F}_T$ and a Bhargava cube $A^{(n)}$ with
\[
m_{ijk} = \sum_{p,q,r=0}^1A^{(n)}_{pqr}u_{ip}^{(n)}v_{jq}^{(n)}w_{kr}^{(n)}, \quad\text{for $(i,j,k)\in\mathcal{I}_n\times\mathcal{J}_n\times\mathcal{K}_n$.}
\]
On $\mathcal{I}_n$, the two paths $u_{i0}^{(n+1)}/u_{i1}^{(n+1)}$ and $u_{i0}^{(n)}/u_{i1}^{(n)}$ give the same set of recurrence relation coefficients $r_i$ in Theorem~\ref{theorem69}, so the itineraries of these two paths coincide. The second path is minimal, so by Theorem~\ref{theoremHD} we can assume that in fact $u_{i0}^{(n+1)}/u_{i1}^{(n+1)}$ is equal to $u_{i0}^{(n)}/u_{i1}^{(n)}$ on $\mathcal{I}_n$, after applying an element of $\text{SL}_2(\mathbb{Z})$ (which changes $A^{(n+1)}$). We define $u_{i0}/u_{i1}$ to be the termwise limiting value of this process, and we define the other paths similarly. Then, from considering the restriction of $\mathbf{M}$ to $\{0,1\}^3$, we can see that the sequence of (adjusted) Bhargava cubes $A^{(n)}$ is constant, giving us the required decomposition.
\end{proof}

Together Theorems~\ref{theorem71} and~\ref{theorem72} imply Theorem~\ref{theoremD}.

\section{Proof of Theorem~\ref{theoremC}}\label{section44}

Theorem~\ref{theoremC} is a stronger version of Theorem~\ref{theoremD} for tame 1-hypertilings. Key to this is the fact that all Bhargava cubes of hyperdeterminant 1 are $\textnormal{SL}_2(\mathbb{Z})^3$-equivalent. To prove this, we define
\[
\mathbf{I}=\mleft(\begin{matrix}1 & 0 \\ 0 & 0 \end{matrix}\ \middle\vert\  \begin{matrix}0 & 0 \\ 0 & 1 \end{matrix}\mright),
\]
which has hyperdeterminant 1.

\begin{theorem}\label{theorem43}
Every Bhargava cube of hyperdeterminant 1 is $\textnormal{SL}_2(\mathbb{Z})^3$-equivalent to $\mathbf{I}$.
\end{theorem}
\begin{proof}
Let $A$ be a Bhargava cube of hyperdeterminant 1. We will first prove that $A$ is $\textnormal{GL}_2(\mathbb{Z})^3$-equivalent to $\mathbf{I}$. The greatest common divisor of the four entries in the top layer of $A$ is 1 because $\Det A=1$. From the Smith normal form for 2-by-2 integer matrices, we can apply an element of $\textnormal{SL}_2(\mathbb{Z})\times\textnormal{SL}_2(\mathbb{Z})\times\{I\}$ (where $I$ is the identity matrix) to reduce the top layer of $A$ to the form
\[
\begin{pmatrix}1 & 0 \\ 0 & \ast \end{pmatrix}.
\]
Then, after applying an element of $\{I\}\times\{I\}\times \textnormal{SL}_2(\mathbb{Z})$, we can reduce $A$ to the form 
\[
\mleft(\begin{matrix}1 & 0 \\ 0 & \gamma \end{matrix}\quad \middle\vert\quad \begin{matrix}0 & \alpha \\ \beta& \delta \end{matrix}\mright),
\]
with hyperdeterminant $\delta^2+4\alpha\beta\gamma=1$. Observe that
\[
\alpha\beta\gamma = \mleft(\frac{1-\delta}{2}\mright)\mleft(\frac{1+\delta}{2}\mright).
\]
The result is easily established if one of $\alpha$, $\beta$, or $\gamma$ is 0, so let us assume that all three are nonzero, in which case $\delta\neq \pm 1$. Let 
\[
Q=\begin{pmatrix*}[r] 1 & 0\\ 0 &\!\!-1\end{pmatrix*}.
\]
After applying a suitable product of the triples $(Q,I,I)$, $(I,Q,I)$, and $(I,I,Q)$, if need be, we can assume that $\alpha$, $\beta$, $\gamma$, and $\delta$ are all negative. We define
\begin{alignat*}{2}
&g_\alpha = \gcd(\alpha,(1-\delta)/2),&& \qquad h_\alpha= -\gcd(\alpha,(1+\delta)/2),\\
&g_\beta = \gcd(\beta,(1-\delta)/2),&& \qquad  h_\beta= -\gcd(\beta,(1+\delta)/2),\\
&g_\gamma = \gcd(\gamma,(1-\delta)/2),&& \qquad  h_\gamma= -\gcd(\gamma,(1+\delta)/2).
\end{alignat*}
Then $g_\alpha,g_\beta,g_\gamma>0$ and $h_\alpha,h_\beta,h_\gamma<0$. Notice that $\alpha/g_\alpha$ and $(1-\delta)/(2g_\alpha)$ are coprime, and 
\[
\mleft(\frac{1-\delta}{2g_\alpha}\mright)\mleft(\frac{1+\delta}{2}\mright)
\]
is divisible by $\alpha/g_\alpha$. Hence $(1+\delta)/2$ is divisible by $\alpha/g_\alpha$, and $\alpha$ is divisible by $\alpha/g_\alpha$ also, so $\alpha/g_\alpha \mid h_\alpha$. Hence $\alpha \mid g_\alpha
h_\alpha$. Now, $(1-\delta)/2$ and $(1+\delta)/2$ are coprime, so $g_\alpha h_\alpha \mid \alpha$. Hence $g_\alpha h_\alpha = \alpha$. Similarly, $g_\beta h_\beta = \beta$ and $g_\gamma h_\gamma = \gamma$.

Next, we have $g_\alpha g_\beta g_\gamma h_\alpha h_\beta h_\gamma= (1-\delta)/2 \times (1+\delta)/2$. Each of $g_\alpha$, $g_\beta$, and $g_\gamma$ is coprime with $(1+\delta)/2$, so we must have
\[
g_\alpha g_\beta g_\gamma = \frac{1-\delta}{2}\quad\text{and}\quad h_\alpha h_\beta h_\gamma = \frac{1+\delta}{2}.
\]
We now define
\[
\renewcommand*{\arraystretch}{1.3}
U_\alpha=\begin{pmatrix}h_\alpha & g_\alpha \\ \frac{\delta-1}{2g_\alpha} & \frac{\delta+1}{2h_\alpha}\end{pmatrix}, \quad
U_\beta=\begin{pmatrix}h_\beta &  g_\beta \\ \frac{\delta-1}{2g_\beta}& \frac{\delta+1}{2h_\beta}\end{pmatrix}, \quad
U_\gamma=\begin{pmatrix}h_\gamma &g_\gamma \\ \frac{\delta-1}{2g_\gamma} & \frac{\delta+1}{2h_\gamma}\end{pmatrix}.
\]
Each of these matrices belongs to $\text{SL}_2(\mathbb{Z})$, and one can check that $(U_\alpha,U_\beta,U_\gamma)$ maps 
\begin{equation*}
\mleft(\begin{matrix}1 & 0 \\ 0 & 0 \end{matrix}\quad \middle\vert\quad \begin{matrix}0 & 0 \\ 0 & 1 \end{matrix}\mright)
\quad\text{to}\quad
\mleft(\begin{matrix}1 & 0 \\ 0 & \gamma \end{matrix}\quad \middle\vert\quad \begin{matrix}0 & \alpha \\ \beta& \delta \end{matrix}\mright). 
\end{equation*}
To summarise, we have found a triple $(V_\alpha,V_\beta,V_\gamma)\in \text{GL}_2(\mathbb{Z})^3$ with $(V_\alpha,V_\beta,V_\gamma)A=\mathbf{I}$. By post-composing with some of the triples $(Q,I,I)$, $(I,Q,I)$, or $(I,I,Q)$ we can find a triple $(W_\alpha,W_\beta,W_\gamma)\in \text{SL}_2(\mathbb{Z})^3$ for which $(W_\alpha,W_\beta,W_\gamma)A$ is equal to either $\mathbf{I}$ or 
\[
\mathbf{I}^\dagger=\mleft(\begin{matrix}1 & 0 \\ 0 & 0 \end{matrix}\ \middle\vert\  \begin{matrix*}[r]0 & 0 \\ 0 & \!\!\!-1 \end{matrix*}\mright).
\]
However, $(J,J,J)\mathbf{I}=\mathbf{I}^\dagger$, where, as usual 
\[
J=\begin{pmatrix*}[r]0 & 1 \\ \!-1 & 0\end{pmatrix*},
\]
so all Bhargava cubes of hyperdeterminant 1 are indeed $\text{SL}_2(\mathbb{Z})^3$-equivalent, as required.
\end{proof}

We can calculate the stabilizer of $\mathbf{I}$ in $\textnormal{SL}_2(\mathbb{Z})^3$ explicitly.

\begin{theorem}\label{theorem62}
The stabilizer of $\mathbf{I}$ in $\textnormal{SL}_2(\mathbb{Z})^3$ is the group
\[
G=\{(I,I,I),(I,-I,-I),(-I,I,-I),(-I,-I,I)\}.
\]
\end{theorem}
\begin{proof}
Clearly $G$ is contained in the stabilizer of $\mathbf{I}$. Conversely, let us consider the equation $(A,B,C)\mathbf{I}=\mathbf{I}$, where $A,B,C\in\text{SL}_2(\mathbb{Z})$. Then
$(A,B,I)\mathbf{I}=(I,I,C^{-1})\mathbf{I}$, which gives
\[
\mleft(\begin{matrix}A_{00}B_{00} & A_{00}B_{10} \\ A_{10}B_{00} & A_{10}B_{10} \end{matrix}\quad \middle\vert\quad \begin{matrix}A_{01}B_{01} & A_{01}B_{11} \\ A_{11}B_{01} & A_{11}B_{11}\end{matrix}\mright)\quad =\quad
\mleft(\begin{matrix}C_{11} & 0 \\  0& -C_{01} \end{matrix}\quad \middle\vert\quad \begin{matrix} -C_{10} &0\\ 0 & C_{00} \end{matrix}\mright).
\]
From the zero entries we see that either $A=\pm I$ and $B=\pm I$ or $A=\pm J$ and $B=\pm J$. In the first case we see from the nonzero entries that $(A,B,C)\in G$. In the second case we obtain
\[
C = \pm\begin{pmatrix}0 & 1 \\ 1 & 0\end{pmatrix},
\]
which is not in $\text{SL}_2(\mathbb{Z})$, so in fact this case cannot arise. Hence the stabilizer of $\mathbf{I}$ is $G$.
\end{proof}

Evidently $G$ is a subgroup of the stabilizer of any Bhargava cube. Some nonsingular cubes have larger stabilizers though. For example, given $A\in\text{SL}_2(\mathbb{Z})$ with trace not $\pm 2$, the stabilizer of the Bhargava cube $(A \mid I)$ contains the triple $(A,(A^{-1})^T,I)$ -- and plenty of other triples besides. A complete understanding of stabilizers in $\text{SL}_2(\mathbb{Z})^3$ can be derived from Bhargava's original work on the Gauss composition law for quadratic forms  \cite{Bh2004}.

\begin{proof}[Proof of Theorem~\ref{theoremC}]
From Theorem~\ref{theoremD} with $\mathbf{I}$ as our choice of Bhargava cube $A$, we see that, given paths $a_i/b_i$, $c_j/d_j$, and $e_k/f_k$ in $\mathscr{F}$, the formula $m_{ijk}=a_ic_je_k+b_id_jf_k$ defines a tame 1-hypertiling. Here $R=S=T=1$ and $\Det A=1$. Clearly this formula is invariant under the group $G\cong \mathbb{Z}_2\times\mathbb{Z}_2$ of Theorem~\ref{theorem62}. Let us now prove that the map is injective. Suppose then that $m_{ijk}=a_i'c_j'e_k'+b_i'd_j'f_k'$, where $a_i'/b_i'$, $c_j'/d_j'$, and $e_k'/f_k'$ are paths in $\mathscr{F}$. Then
\begin{align*}
&\mleft(\begin{pmatrix}a_{i-1} & b_{i-1}\\ a_{i} & b_{i}\end{pmatrix},\begin{pmatrix}c_{j-1} &d_{j-1} \\ c_{j} & d_{j}\end{pmatrix} ,\begin{pmatrix}e_{k-1} & f_{k-1}\\ e_{k} & f_{k}\end{pmatrix}\mright)\mathbf{I}\\
&=\mleft(\begin{pmatrix}a_{i-1}' & b_{i-1}'\\ a_{i}' & b_{i}'\end{pmatrix},\begin{pmatrix}c_{j-1}' & d_{j-1}'\\ c_{j}' & d_{j}'\end{pmatrix} ,\begin{pmatrix}e_{k-1}' & f_{k-1}'\\ e_{k}'& f_{k}'\end{pmatrix}\mright)\mathbf{I}.
\end{align*}
From Theorem~\ref{theorem62}, after applying an element of $G$ to one of the triples of paths, we can ensure that $a_i/b_i=a_i'/b_i'$, $c_j/d_j=c_j'/d_j'$, and $e_k/f_k=e_k'/f_k'$, for $i\in\mathcal{I},j\in\mathcal{J},k\in\mathcal{K}$.

It remains to prove that every tame 1-hypertiling $\mathbf{M}$ arises in this way.  Theorem~\ref{theoremD} tells us that we can express $\mathbf{M}$ in terms of three paths in $\mathscr{F}$ and a Bhargava cube $A$. Then Theorem~\ref{theorem43} allows us to choose $A$ to be $\mathbf{I}$, as required.
\end{proof}

\section{Integer hypertilings with \texorpdfstring{$\text{SL}_2$}{} cross sections}\label{section55}

We finish with a discussion of the positive integer hypertiling $m_{ijk}=F_{2(i+j+k)-1}$ shown in Figure~\ref{figure8} (and conceived in \cite{DePlRuTu2018}), which offers an attractive illustration of topics covered so far. We also consider integer hypertilings with $\text{SL}_2$ cross sections in more generality, without the restriction that all entries are positive.

We recall that $(F_n)_{n\in\mathbb{Z}}$ is the sequence of Fibonacci numbers, where $F_0=0$, $F_1=1$, and $F_{n+1}=F_n+F_{n-1}$, for $n\in\mathbb{Z}$. Observe that, for $n$ odd, $F_{-n}=F_n$ and $F_{n+2}F_{n-2}-F_n^2=1$; the latter follows from a special case of Catalan's identity. 

\begin{lemma}\label{lemma44}
Let $X=\{(a,b)\in\mathbb{N}\times\mathbb{N}: ab \mid a^2+b^2+1\}$. Then 
\[
X=\{(F_{2n-1},F_{2n+1}):n\in\mathbb{Z}\}.
\]
\end{lemma}

This lemma is well-known in mathematical olympiad communities. It can be proved by considering the invertible map $\phi\colon X\longrightarrow X$ defined by $\phi(x,y) = ((x^2+1)/y,x)$ and proving that if $(a,b)\in X$, then $\phi^n(a,b)=(1,1)$, for some integer $n$. We omit the details.

\begin{theorem}\label{theorem183}
All six 2-by-2 cross sections of a positive integer Bhargava cube $A$ have determinant 1 if and only if $A$ is equal to 
\[
A_n=\mleft(\begin{matrix}F_{2n-3} & F_{2n-1} \\ F_{2n-1} & F_{2n+1} \end{matrix}\quad \middle\vert\quad \begin{matrix}F_{2n-1} & F_{2n+1} \\ F_{2n+1} & F_{2n+3} \end{matrix}\mright),
\]
for some integer $n$. 
\end{theorem}
\begin{proof}
It is straightforward to check that every cross section of $A_n$ has determinant~1. For the converse, suppose that the positive integer Bhargava cube
\[
A=\mleft(\begin{matrix}a & b\\ c & d\end{matrix}\ \middle\vert\ \begin{matrix}e & f\\ g & h\end{matrix}\mright)
\]
has the property that every cross section has determinant 1. Observe that
\[
\hspace*{-1.4cm}(a+d)(cf-bg) = (af-be)c-(ag-ce)b-(bh-fd)c+(ch-dg)b=c-b-c+b=0.
\]
Hence $cf=bg$. Therefore \(c = c(af-be)=b(ag-ce)=b\) and $f=g$. Similarly
\[
\hspace*{-1.4cm}(a+f)(bg-de)=(ag-ce)b-(ad-bc)e -(eh-fg)b+(bh-fd)e=b-e-b+e=0. 
\]
Hence $bg=de$. Therefore $e=e(bh-fd)=b(eh-fg)=b$ and $d=g$. We then have $ad-b^2=1$ and $bh-d^2=1$. Consequently, $d\mid b^2+1$ and $b \mid d^2+1$, so we can apply Lemma~\ref{lemma44} to see that $(b,d)=(F_{2n-1},F_{2n+1})$, for some integer $n$. Then $a=(1+b^2)/d=F_{2n-3}$ and $h=(1+d^2)/b=F_{2n+3}$, so $A$ is of the required form.
\end{proof}

It follows immediately from Theorem~\ref{theorem183} that all positive integer hypertilings with $\text{SL}_2$ cross sections have the form $m_{ijk}=F_{2(e+i+j+k)-1}$, for some integer $e$. These are 5-hypertilings, because $\Det A_n=5$.

All the cubes $A_n$ are $\text{SL}_2(\mathbb{Z})^3$ equivalent. Specifically, let 
\[
P=\begin{pmatrix}0& 1\\ -1 & 3\end{pmatrix};
\]
then one can check that $(P^p,P^q,P^r)A_n = A_{n+p+q+r}$. Of course, the property of having $\text{SL}_2$ cross sections is \emph{not} preserved by the action of $\text{SL}_2(\mathbb{Z})^3$, so there are plenty of Bhargava cubes in the $\text{SL}_2(\mathbb{Z})^3$ orbit of $A_n$ without this property.

The 5-hypertiling $m_{ijk}=F_{2(i+j+k)-1}$ can be obtained from three copies of the same path $F_{2n-1}/F_{2n+1}$ in $\mathscr{F}$ (with itinerary $\dotsc,3,3,3,\dotsc$) and the Bhargava cube
\[
\mleft(\begin{matrix*}[r]3& \!\!\!-1\\ \!-1 & 0\end{matrix*}\ \middle\vert\ \begin{matrix*}[r]\!-1 & 0\\ 0 & 1\end{matrix*}\mright).
\]
One can verify this from the Fibonacci identity
\begin{align*}
F_{2(i+j+k)-1}&=3F_{2i-1}F_{2j-1}F_{2k-1}-F_{2i-1}F_{2j-1}F_{2k+1}-F_{2i-1}F_{2j+1}F_{2k-1}\\
&\hspace*{5mm}-F_{2i+1}F_{2j-1}F_{2k-1}+F_{2i+1}F_{2j+1}F_{2k+1},
\end{align*}
which can be checked using Binet's formula.

Let us now turn to more general (not necessarily positive) integer hypertilings with $\text{SL}_2$ cross sections. We define $G$ to be the subgroup of $\text{SL}_2(\mathbb{Z})^3$ generated by $(J,I,I)$, $(I,J,I)$, and $(I,I,J)$, which has order 64. The kernel $K$ of the action of this group on Bhargava cubes is the subgroup $\{(I,I,I),(I,-I,-I),(-I,I,-I),(-I,-I,I)\}$ of order 4, so $G/K$ has order 16. One can check that this action preserves the property of $\text{SL}_2$ cross sections. Hence each of the 16 images of $A_n$ under $G$ is a Bhargava cube with $\text{SL}_2$ cross sections; for example, 
\[
(J,J,I)A_n = \mleft(\begin{matrix} F_{2n+1} & -F_{2n-1}\\ -F_{2n-1} & F_{2n-3} \end{matrix}\quad \middle\vert\quad \begin{matrix}F_{2n+3} & -F_{2n+1} \\-F_{2n+1} & F_{2n-1}  \end{matrix}\mright).
\]
To grasp this orbit of 16 cubes, it helps to observe that, within the orbit, the entry originally at position $(0,0,0)$ is taken to each of the 8 corners of the cube exactly twice, once with each sign.

\begin{theorem}\label{theorem18}
All six 2-by-2 cross sections of a Bhargava cube $A$ have determinant 1 if and only if $A\in G(A_n)$, for some integer $n$.
\end{theorem}
\begin{proof}
We have noted already that all cross sections of elements of $G(A_n)$ have determinant 1. Conversely, suppose that $A$ is a Bhargava cube with the property that every cross section has determinant 1. It is straightforward to verify that no entry of $A$ can be 0.  Consider any face
\[
\begin{pmatrix} a&b\\ c&d\end{pmatrix}
\]
of $A$. Since $ad-bc=1$ we see that $ad$ and $bc$ must have the same sign. It follows that the sign of $a_{ijk}$ is $\pm u^iv^jw^k$, where $u,v,w\in\{-1,1\}$. Applying $(J,I,I)$ to $A$ reverses the sign of $u$, and similarly for $(I,J,I)$ and $v$ and for $(I,I,J)$ and $w$. Also, applying $(-I,I,I)$ reverses the sign of every entry of $A$. Hence there exists $H\in G$ for which $HA$ has positive entries, so $HA=A_n$ for some integer $n$, by Theorem~\ref{theorem183}.
\end{proof}

\subsection*{Acknowledgements}
Short and van Son were supported by EPSRC grant EP/W002817/1, and Zabolotskii was supported by EPSRC DTP grant EP/W524098/1. Short and Zabolotskii acknowledge CIRM and the organisers of the workshop \emph{Frieze patterns in algebra, combinatorics and geometry} (May, 2025), at which the vision for this paper was presented and developed. We thank the referee for discerning suggestions and corrections.

\subsection*{Data availability} 
There is no data associated with this article.

\subsection*{Conflict of interest} 
The authors have no conflict of interest.


\end{document}